\documentclass{mfat}
\pagespan{346}{374}

\usepackage{mmap}
\usepackage[utf8]{inputenc}

\newcommand{\R}{\mathbb{R}}

\newcommand{\N}{\mathbb{N}}

\newcommand{\K}{\mathcal{K}}

\newcommand{\e}{\varepsilon}
\newcommand{\Lb}{\mathcal{L}}

\def\one {1\kern -1mm {\text {\rm I}}}

\newcommand{\esssup}{\mathrm{ess}\sup}

\newcommand{\bfrac}[2]{\genfrac{}{}{0pt}{}{#1}{#2}}

\newtheorem{Theorem}{Theorem}[section]
\newtheorem{Proposition}[Theorem]{Proposition}
\newtheorem{Corollary}[Theorem]{Corollary}
\newtheorem{Lemma}[Theorem]{Lemma}
\newtheorem{Remark}[Theorem]{Remark}
\newtheorem{Definition}[Theorem]{Definition}
\newtheorem{Example}[Theorem]{Example}

\numberwithin{equation}{section}

\begin{document}

\title[Evolution of states and mesoscopic scaling for two-component $\ldots $]
    {Evolution of states and mesoscopic scaling for two-component birth-and-death dynamics in continuum}

%    Information for first author
\author{Martin Friesen}
\address{Department of Mathematics, Bielefeld University, Germany}
%\curraddr{}
\email{mfriesen@math.uni-bielefeld.de}
%\thanks{Thank you.}

%    Information for second author
\author{Oleksandr Kutoviy}
\address{Department of Mathematics, Bielefeld University, Germany}
\email{kutoviy@math.uni-bielefeld.de}
%\thanks{Thank you very much.}

%    General info
\date{30/08/2016;\ \  Revised 01.10.2016.}
%\dedicatory{This paper is dedicated to you.}

\begin{abstract}
  Two coupled spatial birth-and-death Markov evolutions on $\R^d$ are
  obtained as unique weak solutions to the associated Fokker-Planck
  equations. Such solutions are constructed by its associated sequence
  of correlation functions satisfying the so-called
  Ruelle-bound. Using the general scheme of Vlasov scaling we are able
  to derive a system of non-linear, non-local mesoscopic equations
  describing the effective density of the particle system. The results
  are applied to several models of ecology and biology.
\end{abstract}

\maketitle

\section{Introduction}
In the last years we observe an increasing interest to interacting particle systems in the continuum in regard to particular models of ecology, biology and
social sciences, cf. \cite{BCFKKO14,BP99, DL05, FFHKKK15, SEW05}.
An important part of the general theory of interacting particle systems in the continuum is related to the construction of the associated dynamics
and the study of their scalings. In particular, from the point of view of applications it is worthy to investigate the corresponding mesoscopic equations.
Within this framework, each particle is described by its position $x \in \R^d$ and the collection of all particles by the configuration $\gamma$.
Assuming that each two particles cannot occupy the same position and all particles are indistinguishable,
we will choose the configuration space $\Gamma$ to be the collection of all admissible particle configurations $\gamma$, i.e.
\[
 \Gamma = \{ \gamma \subset \R^d \ | \ |\gamma \cap K| < \infty \text{ for all compacts } K \subset \R^d\}.
\]
Here and subsequently $|A|$ denotes the number of elements in the set $A \subset \R^d$.
It is assumed that the microscopic dynamics consists of two elementary events. Namely, the death of a particle $(\gamma \longmapsto \gamma \backslash \{x\})$
and the birth of a particle $(\gamma \longmapsto \gamma \cup \{x\})$. A detailed analysis of such birth-and-death dynamics can be found in \cite{FKK12},
see also the references therein. However, many particular models from ecology and biology require that we have at least two different types of particles.

The aim of this work is to extend the known results for two-component interacting particle systems. Since two particles of different types cannot occupy
the same position, a proper state space for the dynamics is naturally given by the configuration space
\[
 \Gamma^2 := \{ (\gamma^+, \gamma^-) \in \Gamma \times \Gamma\ | \ \gamma^+ \cap \gamma^- = \emptyset \},
\]
cf. \cite{FKO13}. The stochastic time-evolution is incorporated in the particular form of the associated Markov (pre-)generator
$L$. The fact that particles can only die or create new particles leads to the special form of the corresponding generator $L = L^+ + L^-$,
where both operators $L^{\pm}$ are given on polynomially bounded cylinder functions $F$ on $\Gamma^2$ by\vspace*{-1mm}
\begin{align}\label{PHDIPS:04}
 (L^{-}F)(\gamma) = &\ \ \sum _{x \in \gamma^-}d^-(x,\gamma^+, \gamma^- \backslash \{x\})(F(\gamma^+, \gamma^- \backslash \{x\}) - F(\gamma))
 \\ \notag &+ \int _{\R^d}b^-(x,\gamma)(F(\gamma^+, \gamma^- \cup \{x\}) - F(\gamma))\,dx
\end{align}
and

\vspace*{-5mm}
\begin{align}\label{PHDIPS:05}
 (L^{+}F)(\gamma) = &\ \ \sum _{x \in \gamma^+}d^+(x,\gamma^+ \backslash \{x\}, \gamma^-)(F(\gamma^+ \backslash \{x\}, \gamma^-) - F(\gamma))
 \\ \notag &+ \int _{\R^d}b^+(x,\gamma)(F(\gamma^+ \cup \{x\}, \gamma^-) - F(\gamma))\,dx.
\end{align}
For simplicity of notation, we write $\gamma = (\gamma^+, \gamma^-)$, $\gamma^{\pm} \backslash x$ and $\gamma^{\pm} \cup x$ instead of
$\gamma^{\pm} \cup \{x \}$ and $\gamma^{\pm} \backslash \{ x \}$, respectively.
The birth-and-death Markov process associated to the operator $L$, provided it exists, consists of two elementary events.
The death intensities $d^{\pm}(x,\gamma) \geq 0$ determine the probability that the particle $x \in \gamma^{\pm}$ disappears from the configuration $\gamma^{\pm}$.
The birth intensities $b^{\pm}(x,\gamma) \geq 0$ determine the probability for a new particle to appear at $x \in \R^d$.
Multi-species birth-and-death dynamics in the continuum are especially important in ecological, biological and physical applications. For example, the famous
Widom-Rowlinson approach to the phase transition in the continuum is based on a two-component classical gas model. This approach was extended to more general
Potts-type systems (see \cite{GH96}). Recently, the dynamical versions of these models were studied in \cite{FKK15WRMODEL}. The corresponding Markov
statistical dynamics was constructed with the help of Glauber dynamics on $\Gamma$. The dynamical phase transition was shown there in the mesoscopic limit of this dynamics.
In contrast to that work we are able to construct the corresponding dynamics for all $t \geq 0$.
Different kinds of other models, additional comments and explanations for the modeling of tumor growth can be found in \cite{FFHKKK15} and the references therein.

The Markov process associated with the (pre-)generator $L$ may be obtained as the unique solution to certain stochastic differential equations, cf. \cite{GKT06}.
Unfortunately the conditions given in the latter paper are too restrictive for some applications.
A different, functional analytic, approach to the construction of the processes is related to the construction of solutions
to the (backward) Kolmogorov equation on functions $F: \Gamma^2 \longrightarrow \R$
\begin{align}\label{INTRO:00}
 \frac{\partial F_t}{\partial t} = LF_t, \quad F_t|_{t= 0} = F_0, \quad t \geq 0.
\end{align}
Until now, however, there exist up to our knowledge no techniques which can be used to solve \eqref{INTRO:00} in any space of continuous functions on $\Gamma^2$.
The construction of an associated Markov process with general birth-and-death rates is, therefore, still unsolved.
This difficulty comes from the necessity to control the number of particles in any bounded domain in $\R^d$.
One therefore tries to construct an associated evolution of probability measures on $\Gamma^2$, that is the one-dimensional distributions of the Markov process.
The evolution of one-dimensional distributions is an essential problem on its own right.

A probability measure $\mu$ on $\Gamma^2$ describes in this setting the distribution of particles located in $\R^d$.
Functions $F: \Gamma^2 \longrightarrow \R$ are called observables and
\[
 \langle F, \mu\rangle := \int _{\Gamma^2}F(\gamma)\,d\mu(\gamma)
\]
are considered as measurable quantities of the particle system. Above duality yields a weak formulation for the dual equation to \eqref{INTRO:00}, i.e.
\begin{align}\label{INTRO:01}
 \frac{d}{dt} \langle F, \mu_t \rangle = \langle LF, \mu_t\rangle, \quad \mu_t|_{t=0} = \mu_0, \quad t \geq 0.
\end{align}
It is known as the (forward) Kolmogorov equation. In the physical literature \eqref{INTRO:01} is also called Fokker-Planck equation.
Because of the Markovian property of the operator $L$ we expect that solutions to \eqref{INTRO:01}
can be constructed in the class of probability measures on $\Gamma^2$ and hence determine an evolution of states.
Solutions to the Fokker-Planck equation determine the distributions of the associated Markov process, provided, of course, it exists.
Hence $(\mu_t)_{t \geq 0}$ is referred to as the 'statistical description' of the birth-and-death process.

For many, from the point of view of applications, interesting interacting particle systems it seems impossible to solve \eqref{INTRO:01} for any initial state $\mu_0$.
Note that the particular choice $\mu_0 = \delta_{\gamma}$, where $\gamma \in \Gamma^2$, is related with the construction of an associated Markov process on $\Gamma^2$.
It is necessary to restrict ourselves to a certain set of admissible initial states. Here we see a crucial difference comparing with Markov stochastic processes framework
where the evolution can be constructed for any initial data.
One possible class of admissible initial states is given by the collection of all sub-Poissionian states, cf. \cite{KKM08} for the one-component case.
Such states are described in terms of their associated correlation functions $(k_{\mu}^{(n,m)})_{n,m=0}^{\infty}$.
For the convenience of the reader we repeat the definition of correlation functions and their properties in the next section.
The one-component case with general location space $X$ (here $X = \R^d$) was considered in the pioneering works \cite{KK02, L73,L75} to mention some.
The collection of correlation functions $(k_{\mu}^{(n,m)})_{n,m=0}^{\infty}$ is called sub-Poissonian if it satisfies the Ruelle bound
\[
 k_{\mu}^{(n,m)}(x_1, \dots, x_n;y_1, \dots, y_m) \leq A e^{\alpha n}e^{\beta m}, \quad  n,m \in \N_0
\]
for some constants $A > 0$ and $\alpha,\beta \in \R$.
A sub-Poissonian state $\mu$ is therefore a probability measure on $\Gamma^2$ for which the associated correlation functions exist and are sub-Poissonian.
The Cauchy problem \eqref{INTRO:01} is then formally equivalent to a system of Banach space-valued differential equations
\begin{align}\label{PHDIPS:11}
 \frac{\partial k^{(n,m)}_t}{\partial t} = (L^{\Delta}k_t)^{(n,m)}, \quad  k^{(n,m)}_t|_{t=0} = k^{(n,m)}_0, \quad  n,m \in \N_0,
\end{align}
where $L^{\Delta}$ can be seen as an infinite operator-valued matrix. We provide, under some reasonable conditions, existence and uniqueness of weak solutions
to \eqref{PHDIPS:11} and consequently derive existence and uniqueness of weak solutions to \eqref{INTRO:01}.
Note that a solution $(k^{(n,m)}_t)_{n,m=0}^{\infty}$ to \eqref{PHDIPS:11}, in general, does not need to be the correlation function of some state $\mu_t$ on $\Gamma^2$.
For such property additional analysis is required which was only achieved for particular models (see e.g. \cite{KK16, KKM08, KKP08}).

This work is organized as follows. Notations and preliminary
results are introduced in the second section. The third section is
devoted to the construction of solutions to the pre-dual equation
to \eqref{PHDIPS:11}. To this end we first prove that the pre-dual
equation is well-posed and show that solutions can be obtained by
the action of an analytic semigroup of contractions, see Theorem
\ref{PHDIPSTH:01}. The adjoint semigroup is related to existence
and uniqueness of solutions to \eqref{PHDIPS:11}. It is shown that
such solutions are the correlation functions corresponding to an
evolution of states, see Theorem \ref{GMCSTH:01} and Proposition
\ref{PHDIPSTH:02}. The fourth section extends the Vlasov scaling,
cf. \cite{FKK12}, to the case of two-component Markov evolutions,
see Theorem \ref{MCSLTH:00}--\ref{MCSLTH:02}. Examples are
considered in the last section.

\section{Preliminaries}
\subsection*{Space of finite configuration. One-component case}
Denote by $\Gamma_0$ the configuration space of all finite subsets of $\R^d$, i.e.
\[
 \Gamma_0 = \{ \eta \subset \R^d\ | \ |\eta| < \infty \},
\]
where $|\eta|$ denotes the number of particles in the set $\eta$. This space has a natural decomposition into $n$-particle spaces,
$\Gamma_0 = \bigsqcup _{n=0}^{\infty}\Gamma_0^{(n)}$, where $\Gamma_0^{(n)} = \{ \eta \subset \R^d\ | \ |\eta| = n\}, \ \ n \geq 1$
and in the case $n = 0$ we set $\Gamma_0^{(0)} = \{\emptyset\}$.
Denote by $\widetilde{(\R^d)^n}$ the space of all sequences $(x_1, \dots, x_n) \in (\R^d)^n$ with $x_i \neq x_j$ for $i \neq j$.
$\Gamma_0^{(n)}$ can be identified with $\widetilde{(\R^d)^n}$ via the symmetrization map
\[
 \mathrm{sym}_n: \widetilde{(\R^d)^n} \longrightarrow \Gamma_0^{(n)}, \ (x_1, \dots, x_n) \longmapsto \{x_1, \dots, x_n\},
\]
which defines a topology on $\Gamma_0^{(n)}$. Namely, a set $A
\subset \Gamma_0^{(n)}$ is open if and only if
$\mathrm{sym}_n^{-1}(A) \subset \widetilde{(\R^d)^n}$ is open. On
$\Gamma_0$ we define the topology of disjoint unions, i.e. a set
$A \subset \Gamma_0$ is open iff $A \cap \Gamma_0^{(n)}$ is open
in $\Gamma_0^{(n)}$ for all $n \in \N$. Then $\Gamma_0$ is a
locally compact Polish space. Let $\mathcal{B}(\Gamma_0)$ stand
for the Borel-$\sigma$-algebra on $\Gamma_0$. Denote by $dx$ the
Lebesgue measure on $\R^d$ and by $d^{\otimes n} x$ the product
measure on $(\R^d)^n$. Let $d^{(n)}x$ be the image measure
$\Gamma_0^{(n)}$ w.r.t. $\mathrm{sym}_n$. The Lebesgue-Poisson
measure is defined by
\[
 \lambda = \delta_{\emptyset} + \sum _{n=1}^{\infty}\frac{1}{n!}d^{(n)}x.
\]
Given a measurable function $G:\Gamma_0 \times \Gamma_0 \times \Gamma_0 \longrightarrow \R$, then
\begin{align}\label{PHD:00}
 \int _{\Gamma_0}\sum _{\xi \subset \eta}G(\xi, \eta \backslash \xi,\eta)\,d\lambda(\eta) =
 \int _{\Gamma_0}\int _{\Gamma_0}G(\xi, \eta, \eta \cup \xi)\,d\lambda(\xi)d\lambda(\eta)
\end{align}
holds, provided one side of the equality is finite for $|G|$ (see e.g. \cite{KK02}).
For a given measurable function $f: \R^d \longrightarrow \R$ the Lebesgue-Poisson exponential is defined by
\[
 e_{\lambda}(f;\eta) := \prod _{x \in \eta}f(x)
\]
and satisfies the combinatorial formula
\[
 \sum _{\xi \subset \eta}e_{\lambda}(f;\xi) = e_{\lambda}(1+f;\eta).
\]
For computations we will use the identity
\[
 \int _{\Gamma_0}e_{\lambda}(f;\eta)\,d\lambda(\eta) = \exp\Bigg( \int _{\R^d}f(x)\,dx\Bigg),
\]
whenever $f \in L^1(\R^d)$.

\subsection*{Space of finite configurations. Two-component case}
In this part we want to provide a brief overview for the two-component configuration space $\Gamma_0^2$, see \cite{F11TC, FKO13} and the references therein.
We suppose that two different particles cannot occupy the same location $x \in \R^d$ and therefore define the two-component state space by
\[
 \Gamma_0^2 = \{ (\eta^+, \eta^-) \in \Gamma_0 \times \Gamma_0 \ | \ \eta^+ \cap \eta^- = \emptyset \}.
\]
Here and in the following we simply write $\eta$ instead of $(\eta^+, \eta^-) \in \Gamma_0^2$ if no confusion may arise. Set operations
$\xi \subset \eta$, $\xi \cup \eta$ and $\eta \backslash \xi$ are defined component-wise, i.e. by $\xi^{\pm}\subset \eta^{\pm}$, etc.
For $\eta \in \Gamma_0^2$ we let $|\eta| := |\eta^+| + |\eta^-|$. The space $\Gamma_0^2$ has the natural decomposition
\[
 \Gamma_0^2 = \bigsqcup _{n,m=0}^{\infty} \Gamma_0^{(n,m)},
\]
where $\Gamma_0^{(n,m)} = \{ (\eta^+, \eta^-) \subset \R^d \times
\R^d \ | \ \eta^+ \cap \eta^- = \emptyset, \ |\eta^+| = n, \
|\eta^-| = m\}$. The topology on $\Gamma_0^{(n,m)}$ and
$\Gamma_0^2$ is defined in the same way as for $\Gamma_0^{(n)}$
and $\Gamma_0$. It is not difficult to see that this topology is
the same as the subspace topology of the product topology on
$\Gamma_0 \times \Gamma_0$. In particular $\Gamma_0^2$ is a Polish
space. The product measure $\lambda \otimes \lambda$ on $\Gamma_0 \times \Gamma_0$
satisfies
\begin{align}\label{EQ:00}
 \lambda \otimes \lambda( \{ (\eta^+, \eta^-) \in \Gamma_0 \times \Gamma_0 \ | \ \eta^+ \cap \eta^- \neq \emptyset \} ) = 0,
\end{align}
see e.g. \cite{F09}. The Lebesgue-Poisson measure $\lambda^2$ on $\Gamma_0^2$ is
defined as the restriction of $\lambda \otimes \lambda$ to
$\Gamma_0^2$. Since no confusion may arise we use the same
notation $\lambda$ for the Lebesgue-Poisson measure $\lambda^2$ on
$\Gamma_0^2$ and $\lambda$ on $\Gamma_0$. One can show that
integrals w.r.t. integrable functions $G: \Gamma_0^2 \longrightarrow \R$ can be also written as
\begin{align}\label{EQ:01}
 \int _{\Gamma_0^2}G(\eta)\,d\lambda(\eta) = \int _{\Gamma_0}
 \int _{\Gamma_0}G(\eta^+, \eta^-)\,d\lambda(\eta^+)d\lambda(\eta^-).
\end{align}
By \eqref{PHD:00}, \eqref{EQ:00} and, \eqref{EQ:01} we see that the two-component Lebesgue-Poisson measure satisfies
\begin{align}\label{PHD:01}
 \int _{\Gamma_0^2} \sum _{\xi \subset \eta}G(\xi, \eta \backslash \xi,\eta)\,d\lambda(\eta)
 = \int _{\Gamma_0^2}\int _{\Gamma_0^2}G(\xi,\eta,\eta \cup \xi)\,d\lambda(\xi)d\lambda(\eta)
\end{align}
provided one side of the equality is finite for $|G|$.
A set $M \subset \Gamma_0^2$ is called bounded if there exist a compact $\Lambda \subset \R^d$ and $N \in \N_0$ such that
\[
 M \subset \{ (\eta^+, \eta^-) \in \Gamma_0^2 \ | \ \eta^{\pm} \subset \Lambda, \ |\eta| \leq N\}.
\]
A function $G$ is said to have bounded support if it is supported
on a bounded set. Denote by $B_{bs}(\Gamma_0^2)$ the space of all
bounded, measurable functions having bounded support. We say that
$H: \Gamma_0^2 \longrightarrow \R$ is locally integrable if it is
integrable for any bounded set. This is the same as regarding that
the integral $\int
_{\Gamma_0^2}G(\eta)|H(\eta)|\,d\lambda(\eta)$ is finite
for all non-negative functions $G \in B_{bs}(\Gamma_0^2)$.

\subsection*{Space of locally finite configurations}
The one-component configuration space $\Gamma$ consists of all locally finite subsets $\gamma$ of $\mathbb{R}^d$ and it is equipped with the smallest topology
for which $\gamma \longmapsto \sum _{x \in \gamma}f(x)$ is continuous for any continuous function $f: \mathbb{R}^d \longrightarrow \mathbb{R}$ having
compact support. Then $\Gamma$ is a Polish space on which we consider the associated Borel-$\sigma$-algebra.
Given $\alpha \in \R$, the Poisson measure $\pi_{e^{\alpha}}$ on $\Gamma$ is defined as the unique probability measure having the Laplace transform
\[
 \int _{\Gamma}e^{\sum _{x \in \gamma}f(x)}\,d \pi_{e^{\alpha}}(\gamma) = \exp \Bigg( e^{\alpha}\int _{\R^d}(e^{f(x)} - 1)\,dx \Bigg),
\]
see \cite{KK02}.
For two-component systems we consider the product space $\Gamma \times \Gamma$ equipped with the product topology.
For simplicity of notation we write $\gamma := (\gamma^+, \gamma^-)$.
Likewise we use for $\eta \in \Gamma_0^2$ the notation $\eta \subset \gamma$ and $\gamma \backslash \eta$
by which we mean that $\eta^+ \subset \gamma^+, \eta^- \subset \gamma^-$ and $\gamma^+ \backslash \eta^+$, $\gamma^- \backslash \eta^-$.

By \cite{F09} it follows that the product measure $\pi_{e^{\alpha}} \otimes \pi_{e^{\beta}}$ on $\Gamma\times \Gamma$ with $\alpha,\beta \in \mathbb{R}$ is concentrated on
\[
 \Gamma^2 := \{ (\gamma^+, \gamma^-) \in \Gamma \times \Gamma\ | \ \gamma^+ \cap \gamma^- = \emptyset \},
\]
i.e. $\pi_{e^{\alpha}} \otimes \pi_{e^{\beta}}( \Gamma^2 ) = 1$.
Here $\Gamma^2$ is the configuration space for Markov evolutions of particles with two different types (see \cite{FKO13}).
It is equipped with the restriction of the product topology on $\Gamma \times \Gamma$.
The additional restriction is due to the fact that any two particles cannot occupy the same position.
The two-component Poisson measure is defined by $\pi_{e^{\alpha},e^{\beta}} := \pi_{e^{\alpha}}\otimes \pi_{e^{\beta}}|_{\Gamma^2}$.
Below we describe the class of measures for which equation \eqref{INTRO:01} will be considered.

Let $\Lambda_+, \Lambda_- \subset \R^d$ be two compacts, define
\[
 (\Gamma \times \Gamma)_{\Lambda_+, \Lambda_-} = \{ (\gamma^+, \gamma^-) \in \Gamma \times \Gamma \ | \ \gamma^{\pm} \subset \Lambda_{\pm} \},
\]
and set projections $p_{\Lambda^+, \Lambda^-}: \Gamma \times \Gamma \longrightarrow (\Gamma \times \Gamma)_{\Lambda_+, \Lambda_-}, (\gamma^+, \gamma^-) \longmapsto (\gamma^+ \cap \Lambda_+, \gamma^- \cap \Lambda_-)$
A probability measure $\mu$ on $\Gamma \times \Gamma$ is said to be locally absolutely continuous w.r.t. the Poisson measure if
$\mu^{\Lambda_+,\Lambda_-} := \mu \circ p_{\Lambda_+, \Lambda_-}^{-1}$ is absolutely continuous w.r.t. $\pi_{e^{\alpha}} \otimes \pi_{e^{\beta}} \circ p_{\Lambda_+,\Lambda_-}^{-1}$.
Note that this definition is independent of the particular choice of $\alpha$ and $\beta$.
Any such measure $\mu$ satisfies $\mu(\Gamma^2) = 1$ (see \cite{F09}).

For $G \in B_{bs}(\Gamma_0^2)$ define the $K$-transform by
\begin{align}\label{GMCS:08}
 (\mathbb{K}G)(\gamma) = \sum _{\eta \Subset \gamma}G(\eta),
\end{align}
where $\eta \Subset \gamma$ means that the sum only runs over all
finite subsets $\eta$ of $\gamma$. Then $\mathbb{K}G$ is a
polynomially bounded cylinder function, i.e. there exists a
compact $\Lambda \subset \R^d$ and constants $C > 0, \ N \in \N$
such that $(\mathbb{K}G)(\gamma^+, \gamma^-) =
(\mathbb{K}G)(\gamma^+ \cap \Lambda, \gamma^- \cap \Lambda)$ and
\[
 |(\mathbb{K}G)(\gamma)| \leq C(1 + |\gamma^+ \cap \Lambda| + |\gamma^-\cap \Lambda|)^{N}, \quad \gamma \in \Gamma^2
\]
holds. The $\mathbb{K}$-transform $\mathbb{K}: B_{bs}(\Gamma_0^2)
\longrightarrow \mathcal{FP}(\Gamma^2) := K(B_{bs}(\Gamma_0^2))$
is a positivity preserving isomorphism with inverse given by
\[
 (\mathbb{K}^{-1}F)(\eta) := \sum _{\xi \subset \eta}(-1)^{|\eta \backslash \xi|} F(\xi), \quad \eta \in \Gamma_0^2.
\]
Denote by $\mathbb{K}_0$ the restriction of $\mathbb{K}$
determined by evaluating $\mathbb{K}G$ only on $\Gamma_0^2$. Its
inverse is then denoted by $\mathbb{K}_0^{-1}$. Let $\mu$ be locally absolutely continuous w.r.t. the Poisson measure and assume that it has finite local moments, i.e.
\[
 \int _{\Gamma^2}|\gamma^+ \cap \Lambda_+|^n |\gamma^- \cap \Lambda_-|^n \,d \mu(\gamma) < \infty
\]
for all $n \in \N_0$ and all compacts $\Lambda_{\pm} \subset \R^d$. It can be shown that there exists a function $k_{\mu}: \Gamma_0^2 \longrightarrow \R_+$
satisfying the relation
\begin{align}\label{PHDIPS:17}
 \int _{\Gamma^2}\mathbb{K}G(\gamma)\,d\mu(\gamma) = \int _{\Gamma_0^2}G(\eta)k_{\mu}(\eta)\,d \lambda(\eta),
 \quad G \in B_{bs}(\Gamma_0^2),
\end{align}
see \cite{F09}. This function is the so-called correlation function associated with $\mu$.
In such a case the $\mathbb{K}$-transform can be uniquely extended to a bounded linear operator $\mathbb{K}:
L^1(\Gamma_0^2, k_{\mu}d\lambda) \longrightarrow L^1(\Gamma^2,
d\mu)$ such that $\Vert \mathbb{K} G \Vert_{L^1(\Gamma^2, d\mu)}
\leq \Vert G \Vert_{L^1(\Gamma_0^2, k_{\mu} d\lambda)}$ and
\eqref{GMCS:08} holds for $\mu$-a.a. $\gamma \in \Gamma^2$ (repeat e.g. the arguments given in \cite{KK02}).
Let $\Lb_{k_{\mu}} := L^1(\Gamma_0^2, k_{\mu}d\lambda)$, for
$k_{\mu}(\eta) := e^{\alpha|\eta^+|}e^{\beta|\eta^-|}$ we also
write $\Lb_{k_{\mu}} \equiv \Lb_{\alpha,\beta}$. A function $G \in
\Lb_{\alpha,\beta}$ is called positive definite if $\mathbb{K}G
\geq 0$. Denote by $\Lb_{\alpha,\beta}^+$ the cone of all positive
definite functions.

The duality $\langle G, k\rangle := \int
_{\Gamma_0^2}G(\eta)k(\eta)\,d\lambda(\eta)$ is used for
the identification $\Lb_{\alpha,\beta}^* \cong \K_{\alpha,\beta}$.
Here $\mathcal{K}_{\alpha,\beta}$ stands for the Banach space of
all equivalence classes of functions $k: \Gamma_0^2
\longrightarrow \R$ equipped with norm
\[
 \Vert k \Vert_{\K_{\alpha,\beta}} = \esssup _{\eta \in \Gamma_0^2}\ |k(\eta)|e^{-\alpha|\eta^+|}e^{-\beta|\eta^-|}.
\]
A function $k \in \K_{\alpha,\beta}$ is said to be positive definite if $\langle G, k \rangle \geq 0$ holds for all
$G \in B_{bs}^+(\Gamma_0^2) := B_{bs}(\Gamma_0^2) \cap \Lb_{\alpha,\beta}^+$. Let $\K_{\alpha,\beta}^+$ be the space of all positive definite functions in $\K_{\alpha,\beta}$.
Denote by $\mathcal{P}_{\alpha,\beta}$ the collection of all probability measures having finite local moments and being locally absolutely continuous
w.r.t. the Poisson measure such that for each $\mu \in \mathcal{P}_{\alpha, \beta}$ its correlation function $k_{\mu}$ belongs to $\mathcal{K}_{\alpha,\beta}$.
\begin{Theorem}\label{GMCSTH:05}\cite{KK02, L73,L75}
 The following assertions are satisfied.
 \begin{enumerate}
  \item Let $\mu \in \mathcal{P}_{\alpha,\beta}$ with correlation function $k_{\mu}$. Then $k_{\mu}(\emptyset) = 1$ and $k_{\mu}$ is positive definite.
  \item Conversely, let $k \in \K_{\alpha,\beta}^+$ and assume that $k(\emptyset) = 1$ holds.
  Then there exists a unique $\mu \in \mathcal{P}_{\alpha,\beta}$ with $k$ as its correlation function.
 \end{enumerate}
\end{Theorem}
Let $\mu \in \mathcal{P}_{\alpha,\beta}$, then by Fubini's theorem
\[
 \int _{\Gamma^2}\mathbb{K}G(\gamma)\,d\mu(\gamma) = \int _{\Gamma}\int _{\Gamma}
 \mathbb{K}G(\gamma^+, \gamma^-)\,d\mu(\gamma^+,\gamma^-)
\]
holds for all $G \in \Lb_{\alpha,\beta}$.

\section{Construction of dynamics}
Let $L = L^{-} + L^{+}$ be given by \eqref{PHDIPS:04} and \eqref{PHDIPS:05}.
We suppose that the intensities satisfy the condition given below.
\begin{enumerate}
 \item[(A)] For all $x \in \R^d$
 \begin{align}\label{PHDIPS:27}
  \R^d \times \Gamma^2 \ni (x,\gamma) \longmapsto d^+(x, \gamma^+ \backslash x, \gamma^-), d^-(x, \gamma^+, \gamma^- \backslash x), b^+(x,\gamma), b^-(x,\gamma) \in [0, \infty]
 \end{align}
 are measurable w.r.t the product Borel-$\sigma$-algebras on $\R^d \times \Gamma^2$ and $\mathcal{B}(\overline{\R})$, where $\overline{\R} = \R \cup \{ \infty\}$.
 Moreover, their restrictions to $\R^d \times \Gamma_0^2$ take values in $[0,\infty)$ and, for any compact $\Lambda \subset \R^d$ and bounded set $M \subset \Gamma_0^2$
 \begin{align}\label{PHDIPS:21}
  \int _{\Lambda}\int _{M}(d^+(x,\eta) + d^-(x,\eta) + b^+(x,\eta)
  + b^-(x,\eta))\,d\lambda(\eta)dx < \infty.
 \end{align}
\end{enumerate}

\subsection{Quasi-observables}
Introduce the cumulative death intensity by
\[
 M(\eta) := \sum _{x \in \eta^-}d^-(x,\eta^+, \eta^- \backslash x) + \sum _{x \in \eta^+}d^+(x,\eta^+ \backslash x, \eta^-)
\]
and set $c(L, \alpha,\beta;\eta) = c(\eta) = c(\alpha,\beta;\eta)$ by
\begin{align}\label{EQ:02}
 c(L, \alpha,\beta;\eta) := &\ \sum _{x \in \eta^-}\int _{\Gamma_0^2}e^{\alpha|\xi^+|}e^{\beta|\xi^-|}
 |\mathbb{K}_0^{-1}d^-(x,\cdot \cup \eta^+, \cdot \cup \eta^- \backslash x)|(\xi)\,d\lambda(\xi)
  \\ \notag &+ \sum _{x \in \eta^+}\int _{\Gamma_0^2}e^{\alpha|\xi^+|}e^{\beta|\xi^-|}|
  \mathbb{K}_0^{-1}d^+(x,\cdot \cup \eta^+ \backslash x, \cdot \cup \eta^-)|(\xi)\,d\lambda(\xi)
  \\ \notag &+ e^{-\beta}\sum _{x \in \eta^-}\int _{\Gamma_0^2}e^{\alpha|\xi^+|}e^{\beta|\xi^-|}
  |\mathbb{K}_0^{-1}b^-(x,\cdot \cup \eta^+, \cdot \cup \eta^-\backslash x)|(\xi)\,d \lambda(\xi)
  \\ \notag &+ e^{-\alpha}\sum _{x \in \eta^+}\int _{\Gamma_0^2}e^{\alpha|\xi^+|}e^{\beta|\xi^-|}
  |\mathbb{K}_0^{-1}b^+(x,\cdot \cup \eta^+ \backslash x, \cdot \cup \eta^-)|(\xi)\,d\lambda(\xi).
\end{align}
Define a linear mapping on $B_{bs}(\Gamma_0^2)$ by $\widehat{L} := \mathbb{K}_0^{-1}L\mathbb{K}_0$. Using the methods proposed
in \cite{FKK12, FKO13} we can compute $\widehat{L}$. It has the
form $\widehat{L} = A + B$. The latter expressions are given by $(AG)(\eta) = - M(\eta)G(\eta)$ and by
\begin{align*}
 (BG)(\eta) =  &- \sum _{\xi \subsetneq \eta}G(\xi) \sum _{x \in \xi^-} (\mathbb{K}_0^{-1}d^-(x,\cdot \cup \xi^+, \cdot \cup \xi^- \backslash x))(\eta \backslash \xi)
 \\ &- \sum _{\xi \subsetneq \eta}G(\xi) \sum _{x \in \xi^+} (\mathbb{K}_0^{-1}d^+(x,\cdot \cup \xi^+ \backslash x, \cdot \cup \xi^-))(\eta \backslash \xi)
 \\ &+ \sum _{\xi \subset \eta} \int _{\R^d}G(\xi^+, \xi^- \cup x)
 (\mathbb{K}_0^{-1}b^-(x,\cdot \cup \xi^+, \cdot \cup \xi^-))(\eta \backslash \xi)\,dx
 \\ &+ \sum _{\xi \subset \eta} \int _{\R^d}G(\xi^+ \cup x, \xi^-)
 (\mathbb{K}_0^{-1}b^+(x,\cdot \cup \xi^+,\cdot \cup \xi^-))(\eta \backslash \xi)\,dx.
\end{align*}
Below we consider the linear mapping $\widehat{L}$ in one fixed Banach space and construct a semigroup associated to the Cauchy problem for quasi-observables
\begin{align}\label{GMCS:48}
 \frac{\partial G_t}{\partial t} = \widehat{L}G_t, \quad G_t|_{t=0} = G_0 \in \Lb_{\alpha,\beta}.
\end{align}
Note that solutions to \eqref{INTRO:00} are formally related to
\eqref{GMCS:48} by $F_t = \mathbb{K}G_t$. Denote by $\one^*$ the function
given by $\one^*(\eta) := 0^{|\eta|} = \begin{cases}1, & |\eta| =
0
\\ 0, & \text{ otherwise}\end{cases}$.
\begin{Theorem}\label{PHDIPSTH:01}
 Suppose that (A) is satisfied and assume that there exists $\alpha, \beta \in \R$ and a constant $a = a(\alpha,\beta) \in (0,2)$ such that
 \begin{align}\label{PHDIPS:02}
  c(\alpha,\beta;\eta) \leq a(\alpha,\beta)M(\eta), \quad \eta \in \Gamma_0^2
 \end{align}
 holds. Then the following assertions are true:
 \begin{enumerate}
  \item[(a)] The closure of $(\widehat{L}, B_{bs}(\Gamma_0^2))$ is the generator of an analytic semigroup $(\widehat{T}_{\alpha,\beta}(t))_{t \geq 0}$
   of contractions on $\Lb_{\alpha,\beta}$ such that $\widehat{T}_{\alpha,\beta}(t)\one^* = \one^*$ holds. Moreover, the closure is given by
   \[
    D_{\alpha,\beta}(\widehat{L}) = \{ G \in \Lb_{\alpha,\beta} \ | \ M \cdot G \in \Lb_{\alpha,\beta}\}
   \]
   with $\widehat{L} = A + B$ defined as above.
  \item[(b)] Suppose that \eqref{PHDIPS:02} holds for all $\alpha \in (\alpha_*, \alpha^*)$ and $\beta \in (\beta_*, \beta^*)$ with
  $\alpha_* < \alpha^*$ and $\beta_* < \beta^*$, and possibly different constants $a = a(\alpha,\beta)$.
  Then, for any $\alpha' < \alpha$ and $\beta' < \beta$ the space $\Lb_{\alpha,\beta}$ is invariant for $\widehat{T}_{\alpha',\beta'}(t)$
  and $\widehat{T}_{\alpha,\beta}(t) = \widehat{T}_{\alpha',\beta'}(t)|_{\Lb_{\alpha,\beta}}$ holds.
 \end{enumerate}
\end{Theorem}
\begin{proof}
 (a) Set $D_{\alpha,\beta}(A) := \{ G \in \Lb_{\alpha,\beta}\ | \ M \cdot G \in \Lb_{\alpha,\beta}\}$.
 Then, since $M \geq 0$, the operator $(A, D_{\alpha,\beta}(A))$ is the generator of an analytic (of angle $\frac{\pi}{2}$), positive $C_0$-semigroup
 $(e^{-tM})_{t \geq 0}$ on $\Lb_{\alpha,\beta}$, see \cite{ENG00}.
 Let $B'$ be defined for any $G \in B_{bs}(\Gamma_0^2)$ by
 \begin{align*}
  (B'G)(\eta^+, \eta^-) :=  &\ \sum _{\xi \subsetneq \eta}G(\xi) \sum _{x \in \xi^-} |\mathbb{K}_0^{-1}d^-(x,\cdot \cup \xi^+, \cdot \cup \xi^- \backslash x)|(\eta \backslash \xi)
  \\ &+ \sum _{\xi \subsetneq \eta}G(\xi) \sum _{x \in \xi^+} |\mathbb{K}_0^{-1}d^+(x,\cdot \cup \xi^+ \backslash x, \cdot \cup \xi^-)|(\eta \backslash \xi)
  \\ &+ \sum _{\xi \subset \eta} \int _{\R^d}G(\xi^+, \xi^- \cup x) |\mathbb{K}_0^{-1}b^-(x,\cdot \cup
  \xi^+, \cdot \cup \xi^-)|(\eta \backslash \xi)\,dx
  \\ &+ \sum _{\xi \subset \eta} \int _{\R^d}G(\xi^+ \cup x, \xi^-) |\mathbb{K}_0^{-1}b^+(x,\cdot
  \cup \xi^+,\cdot \cup \xi^-)|(\eta \backslash \xi)\,dx.
 \end{align*}
 Fix $r \in (0,1)$, cf. \eqref{PHDIPS:02}, such that $a(\alpha,\beta) < 1 + r < 2$.
 For each $0 \leq G \in D_{\alpha,\beta}(A)$, see \eqref{PHD:01}, we obtain by \eqref{PHDIPS:02}
 \begin{align*}
  \int _{\Gamma_0^2}B'G(\eta)e^{\alpha|\eta^+|}e^{\beta|\eta^-|}\,d\lambda(\eta)
  &= \int _{\Gamma_0^2}(c(\alpha,\beta; \eta) - M(\eta))G(\eta)e^{\alpha|\eta^+|}e^{\beta|\eta^-|}\,d\lambda(\eta)
  \\ &\leq r \int _{\Gamma_0^2}M(\eta)G(\eta)e^{\alpha|\eta^+|}e^{\beta|\eta^-|}\,d\lambda(\eta)
 \end{align*}
 and hence $\int _{\Gamma_0^2}\left(A + \frac{1}{r}B'\right)G(\eta)e^{\alpha|\eta^+|}e^{\beta|\eta^-|}
 \,d\lambda(\eta) \leq 0$ holds.
 Therefore by \cite[Theorem~2.2]{TV06} the operator $(A + B', D_{\alpha,\beta}(A))$ is the generator of a strongly continuous semigroup $(U(s))_{s \geq 0}$
 of contractions which preserves positivity. By \cite[Theorem 1.1, Theorem 1.2]{AR91} the operator $(A + B, D_{\alpha,\beta}(A))$ is the generator
 of an analytic semigroup $(\widehat{T}(s))_{s \geq 0}$ such that $|\widehat{T}(s)G| \leq U(s)|G|$ holds for all $G \in \Lb_{\alpha,\beta}$.
 Since $U(s)$ is a contraction operator, so is $\widehat{T}(s)$. It remains to show that the closure of $(\widehat{L}, B_{bs}(\Gamma_0^2))$
 is given by $(A + B, D_{\alpha,\beta}(A))$. To this end it suffices to show that $B_{bs}(\Gamma_0^2)$ is a core for $A+B$.

 Let $G \in D_{\alpha,\beta}(A)$, $A_n \subset \Gamma_0^2$ an increasing sequence of bounded sets with $\underset{n \geq 1}{\bigcup}\ A_n = \Gamma_0^2$ and let
 $G_n(\eta) := \one_{A_n}(\eta)\one_{|G|\leq n}(\eta)G(\eta)$. Then $|G_n| \leq |G|$ and $G_n \longrightarrow G$ almost everywhere. Hence, $M\cdot G_n \longrightarrow M\cdot G$
 and by dominated convergence also $BG_n \longrightarrow BG$, as $n \to \infty$ a.e. As a consequence  $\widehat{L}G_n \longrightarrow \widehat{L}G$,
 as $n \to \infty$ almost everywhere. By
 \[
  |\widehat{L}G_n| \leq M|G_n| + B'|G_n| \leq (M+B')|G| \in \Lb_{\alpha,\beta}
 \]
 and dominated convergence we obtain $\widehat{L}G_n \longrightarrow \widehat{L}G$ in $\Lb_{\alpha,\beta}$. Therefore $B_{bs}(\Gamma_0^2) \subset D_{\alpha,\beta}(\widehat{L})$
 is dense in the graph norm, i.e. the closure of $(\widehat{L}, B_{bs}(\Gamma_0^2))$ is given by $(A+B, D_{\alpha,\beta}(A))$. \newline

 (b) Let $\alpha' < \alpha$ and $\beta' < \beta$ such that
\eqref{PHDIPS:02} also holds for $(\alpha',\beta')$. Denote by
$(\widehat{T}_{\alpha',\beta'}(s))_{s \geq 0}$
 the corresponding semigroup on $\Lb_{\alpha',\beta'}$ with generator $(\widehat{L}, D_{\alpha',\beta'}(\widehat{L}))$ on the domain
 \[
  D_{\alpha',\beta'}(\widehat{L}) := \{ G \in \Lb_{\alpha',\beta'} \ | \ M \cdot G \in \Lb_{\alpha',\beta'}\}.
 \]
 We have to show that $\Lb_{\alpha,\beta}$ is invariant for $\widehat{T}_{\alpha',\beta'}(s)$ and
 \begin{align}\label{GMCS:01}
  \widehat{T}(s)G = \widehat{T}_{\alpha', \beta'}(s)G, \quad G \in \Lb_{\alpha,\beta}, \quad s \geq 0.
 \end{align}
 To this end we define a linear isomorphism
 \[
  S: \Lb_{\alpha,\beta} \longrightarrow \Lb_{\alpha',\beta'}, \quad (SG)(\eta) = e^{(\alpha-\alpha')|\eta^+|}e^{(\beta-\beta')|\eta^-|}G(\eta)
 \]
 with inverse $S^{-1}$ given by $(S^{-1}G)(\eta) = e^{-(\alpha - \alpha')|\eta^+|}e^{-(\beta - \beta')|\eta|} G(\eta)$.
 Define on $\Lb_{\alpha',\beta'}$ a new operator by $\widehat{L}_1 := S \widehat{L} S^{-1}$ equipped with the domain
 \[
  D_{\alpha',\beta'}(\widehat{L}_1) = \{ G \in \Lb_{\alpha',\beta'}\ | \ S^{-1}G \in D_{\alpha,\beta}(\widehat{L}) \} = \{ G \in \Lb_{\alpha',\beta'} \ | \ M S^{-1}G \in \Lb_{\alpha,\beta}\}.
 \]
 Since $\Vert M S^{-1}G \Vert_{\Lb_{\alpha,\beta}} = \Vert MG \Vert_{\Lb_{\alpha',\beta'}}$ we obtain
 $D_{\alpha',\beta'}(\widehat{L}_1) = D_{\alpha',\beta'}(\widehat{L})$.
 Let us show that $(\widehat{L}_1, D_{\alpha',\beta'}(\widehat{L}_1))$ is the generator of a $C_0$-semigroup on $\Lb_{\alpha',\beta'}$.
 The definition of $S$ and $S^{-1}$ implies $\widehat{L}_1 = A + B_1$ where $A$ is the same as for $\widehat{L}$ and $B_1$ is given by
 \begin{align*}
 &\ (B_1G)(\eta) =
 \\ &- \sum _{\xi \subsetneq \eta}G(\xi)e^{(\alpha - \alpha')|\eta^+ \backslash \xi^+|}e^{(\beta - \beta')|\eta^- \backslash \xi^-|} \sum _{x \in \xi^-}  (\mathbb{K}_0^{-1}d^-(x,\cdot \cup \xi^+, \cdot \cup \xi^- \backslash x))(\eta \backslash \xi)
 \\ &- \sum _{\xi \subsetneq \eta}G(\xi)e^{(\alpha - \alpha')|\eta^+ \backslash \xi^+|}e^{(\beta - \beta')|\eta^- \backslash \xi^-|} \sum _{x \in \xi^+} (\mathbb{K}_0^{-1}d^+(x,\cdot \cup \xi^+ \backslash x, \cdot \cup \xi^-))(\eta \backslash \xi)
 \\ &+ e^{-(\beta - \beta')}\sum _{\xi \subset \eta}
 \int _{\R^d}G(\xi^+, \xi^- \cup x)e^{(\alpha - \alpha')|\eta^+ \backslash \xi^+|}e^{(\beta - \beta')
 |\eta^- \backslash \xi^-|} (\mathbb{K}_0^{-1}b^-(x,\cdot \cup \xi))(\eta \backslash \xi)\,dx
 \\ &+ e^{-(\alpha - \alpha')}\sum _{\xi \subset \eta} \int _{\R^d}G(\xi^+ \cup x,
 \xi^-)e^{(\alpha - \alpha')|\eta^+ \backslash \xi^+|}e^{(\beta - \beta')|\eta^- \backslash \xi^-|}
 (\mathbb{K}_0^{-1}b^+(x,\cdot \cup \xi))(\eta \backslash \xi)\,dx.
 \end{align*}
 Define analogously to $B'$ the positive operator $B_1'$ such that $|B_1G| \leq B_1'|G|$, then for any non-negative function $G \in D_{\alpha',\beta'}(\widehat{L}_1)$
 we obtain
 \[
  \int _{\Gamma_0^2}B_1'G(\eta) e^{\alpha'|\eta^+|}e^{\beta'|\eta^-|}\,d\lambda(\eta) =
  \int _{\Gamma_0^2}(c(\alpha,\beta;\eta) - M(\eta))G(\eta)e^{\alpha'|\eta^+|}e^{\beta'|\eta^-|}\,d\lambda(\eta).
 \]
 The same arguments as for the construction of $\widehat{T}_{\alpha,\beta}(t)$ show that $(A + B_1', D_{\alpha',\beta'}(\widehat{L}_1))$ is the generator of a sub-stochastic semigroup
 and hence $(\widehat{L}_1, D_{\alpha',\beta'}(\widehat{L}_1))$ is the generator of a $C_0$-semigroup. Now \cite[Chapter 4, Theorem 5.5, Theorem 5.8]{PAZ83} implies that
 $\Lb_{\alpha,\beta}$ is invariant for $\widehat{T}_{\alpha',\beta'}(t)$ and the restriction to $\Lb_{\alpha,\beta}$ is a $C_0$-semigroup given by
 $\widetilde{T}_{\alpha,\beta}(t) := \widehat{T}_{\alpha',\beta'}(t)|_{\Lb_{\alpha,\beta}}$.
 The generator of $\widetilde{T}_{\alpha,\beta}(t)$ is given by the part of $(\widehat{L}, D_{\alpha',\beta'}(\widehat{L}))$ in $\Lb_{\alpha,\beta}$,
 that is by
 \begin{align*}
  D_{\alpha',\beta'}(\widehat{L})|_{\Lb_{\alpha,\beta}}:= & \{ G \in D_{\alpha',\beta'}(\widehat{L})
  \cap \Lb_{\alpha,\beta} \ | \ \widehat{L}G \in \Lb_{\alpha,\beta}\}
  \\ = &  \{ G \in \Lb_{\alpha,\beta}\ | \ M\cdot G \in \Lb_{\alpha',\beta'}, \ \widehat{L}G \in \Lb_{\alpha,\beta}\}.
 \end{align*}
 Condition \eqref{PHDIPS:02} therefore implies $D_{\alpha,\beta}(\widehat{L}) \subset  D_{\alpha',\beta'}(\widehat{L})|_{\Lb_{\alpha,\beta}}$ and hence
 $(\widehat{L},  D_{\alpha',\beta'}(\widehat{L})|_{\Lb_{\alpha,\beta}})$
 is an extension of $(\widehat{L}, D_{\alpha,\beta}(\widehat{L}))$. Denote by $R(\lambda;\widehat{L})$ the resolvent for $(\widehat{L}, D_{\alpha,\beta}(\widehat{L}))$
 and by $\widetilde{R}(\lambda;\widehat{L})$ the resolvent for $(\widehat{L},  D_{\alpha',\beta'}(\widehat{L})|_{\Lb_{\alpha,\beta}})$. For sufficiently large $\lambda > 0$
 it follows that $R(\lambda,\widehat{L})G \in D_{\alpha,\beta}(\widehat{L}) \subset  D_{\alpha',\beta'}(\widehat{L})|_{\Lb_{\alpha,\beta}}$ for any $G \in \Lb_{\alpha,\beta}$ and thus
 \[
  \widetilde{R}(\lambda;\widehat{L})G - R(\lambda;\widehat{L})G = \widetilde{R}(\lambda;\widehat{L})((\lambda - \widehat{L}) - (\lambda - \widehat{L}))R(\lambda;\widehat{L})G = 0,
 \]
 where we have used that for elements in $D_{\alpha,\beta}(\widehat{L})$ the action of the generators is given by the formulas for $\widehat{L} = A + B$ and hence coincide.
\end{proof}
For one-component models, i.e. $b^- = 0 = d^-$, a similar construction was already done in \cite{FKK12}.
The main assumption was that each term in $c(\alpha,\beta;\eta)$ is bounded by $\frac{3}{2}M(\eta)$ and
it was not clear whether $\widehat{T}_{\alpha,\beta}(t)$ is a contraction operator for $t \geq 0$.
The next example shows that the constant $2$ in \eqref{PHDIPS:02} is optimal.
\begin{Example}
 Take $d^- = 1$, $b^- = z > 0$ constant and $b^+ = d^+ = 0$, then condition \eqref{PHDIPS:02} can be restated to $z < e^{\beta}$ and $\alpha \in \R$ is arbitrary.
The evolution equation \eqref{GMCS:48} is in this case exactly solvable and hence has for every initial condition $G_0$ the solution $(G_t)_{t \geq 0}$ given by
\[
 G_t(\eta) = e^{-t|\eta^-|} \int _{\Gamma_0}G(\eta^+, \eta^- \cup \xi^-) e_{\lambda}\left( z (1 - e^{-t});
 \xi^-\right) \,d\lambda(\xi^-), \quad \eta \in \Gamma_0^2,
\]
see \cite{F11} for the one-component case. If condition \eqref{PHDIPS:02} is satisfied, then $G_t \in \Lb_{\alpha,\beta}$.
Suppose that $a(\alpha,\beta) > 2$, i.e. $z > e^{\beta}$ and let $t_0 > 0$ such that $(1-e^{-t})z > e^{\beta}$ for all $t \geq t_0$ and hence
\[
 z(1 - e^{-t})(e^{-\beta} + e^{-t}) \geq z (1-e^{-t}) e^{-\beta} > 1.
\]
Take $0 \leq G \in \Lb_{\alpha,\beta}$ such that $G \not \in \Lb_{\alpha, \beta'}$ for any $\beta' > \beta$. The unique solution $G_t$ is then
positive and satisfies
\begin{align*}
 \Vert G_t \Vert_{\Lb_{\alpha,\beta}} &= \int _{\Gamma_0^2}\int _{\Gamma_0}e^{-t|\eta^-|} G(\eta^+,
 \eta^- \cup \xi^-) e_{\lambda}(z (1-e^{-t});\xi^-)e^{\alpha|\eta^+|}e^{\beta|\eta^-|}\,d\lambda(\xi^-)d\lambda(\eta)
 \\ &= \int _{\Gamma_0^2} G(\eta^+, \xi^-) \left( (e^{-\beta} + e^{-t})z(1 - e^{-t}) \right)^{|\xi^-|}
 e^{\alpha|\eta^+|} e^{\beta|\xi^-|}\,d\lambda(\eta^+, \xi^-) = \infty.
\end{align*}
\end{Example}
Below we show that $\widehat{T}_{\alpha,\beta}(t)$ depends continuously (in a certain sense) on the birth-and-death rates.
Let $d_n^{\pm}, d^{\pm}$ and $b_n^{\pm},b^{\pm}$ be such that condition (A) is satisfied. Let
\begin{align*}
  c_n(\alpha,\beta;\eta) &= \sum _{x \in \eta^-}\int _{\Gamma_0^2}e^{\alpha|\xi^+|}e^{\beta|\xi^-|}| \mathbb{K}_0^{-1}(d^- - d_n^-)(x,\cdot \cup \eta^+, \cdot \cup \eta^- \backslash x)|(\xi)\,d \lambda(\xi)
 \\ &\ \ + \sum _{x \in \eta^+}\int _{\Gamma_0^2}e^{\alpha|\xi^+|}e^{\beta|\xi^-|}| \mathbb{K}_0^{-1}(d^+ - d_n^+)(x,\cdot \cup \eta^+\backslash x, \cdot \cup \eta^-)|(\xi)\,d \lambda(\xi)
 \\ &\ \ + e^{-\beta}\sum _{x \in \eta^-}\int _{\Gamma_0^2}e^{\alpha|\xi^+|}e^{\beta|\xi^-|}| \mathbb{K}_0^{-1}(b^- - b_n^-)(x,\cdot \cup \eta^+, \cdot \cup \eta^- \backslash x)|(\xi)\,d \lambda(\xi)
 \\ &\ \ + e^{-\alpha}\sum _{x \in \eta^+}\int _{\Gamma_0^2}e^{\alpha|\xi^+|}e^{\beta|\xi^-|}| \mathbb{K}_0^{-1}(b^+ - b_n^+)(x,\cdot \cup \eta^+\backslash x, \cdot \cup \eta^-)|(\xi)\,d \lambda(\xi)
\end{align*}
and $M_n(\eta) = \sum _{x \in \eta^-}d_n^-(x,\eta^+, \eta^- \backslash x) + \sum _{x \in \eta^+}d_n^+(x,\eta^+ \backslash x, \eta^-)$.
\begin{Theorem}
 Suppose that the following conditions are satisfied.
 \begin{enumerate}
  \item There exist $\alpha,\beta \in \R$ and a constant $a = a(\alpha,\beta) \in (0,2)$ (independent of $n$) such that
  \[
   c_n(\alpha,\beta;\eta) \leq a(\alpha,\beta)M_n(\eta), \quad \eta \in \Gamma_0^2, \quad n \geq 1
  \]
  holds.
  \item There exist constants $C > 0$, $N \in \N$ and $\tau \geq 0$ such that
  \[
   d^-_n(x,\eta) + d^+_n(x,\eta) \leq C(1+|\eta|)^Ne^{\tau |\eta|}, \quad \eta \in \Gamma_0^2, \quad  x \in \R^d
  \]
  holds.
  \item $c_n(\alpha,\beta;\eta) \longrightarrow 0$, $n \to \infty$ holds for all $\eta \in \Gamma_0^2$.
 \end{enumerate}
 Then \eqref{PHDIPS:02} is satisfied. Let $\widehat{T}_{\alpha,\beta}(t)$ and $\widehat{T}_{\alpha,\beta}^n(t)$ be the associated semigroups on $\Lb_{\alpha,\beta}$.
 Then $\widehat{T}_{\alpha,\beta}^n(t)G \longrightarrow \widehat{T}_{\alpha,\beta}(t)G$, as $n \to \infty$ in $\Lb_{\alpha,\beta}$ uniformly on compacts for $t \geq 0$
 and all $G \in \Lb_{\alpha,\beta}$.
\end{Theorem}
\begin{proof}
 Denote by $c(L,\alpha,\beta;\eta)$ the function defined in \eqref{EQ:02}. Let $c(L_n,\alpha,\beta;\eta)$ be given as in \eqref{EQ:02} with $d^{\pm},b^{\pm}$
 replaced by $d_n^{\pm},b_n^{\pm}$. It is not difficult to see that
 \[
  |c(L_n,\alpha,\beta;\eta) - c(L,\alpha,\beta;\eta)| \leq c_n(\alpha,\beta;\eta)
 \]
 and $|M_n(\eta) - M(\eta)| \leq c_n(\alpha,\beta;\eta)$ holds. Thus, for any $\eta \in \Gamma_0^2$
 \[
  c(L,\alpha,\beta;\eta) = \lim _{n \to \infty}\ c(L_n,\alpha,\beta;\eta) \leq a(\alpha,\beta) \lim _{n\to \infty}\ M_n(\eta) = a(\alpha,\beta)M(\eta)
 \]
 and hence \eqref{PHDIPS:02} holds. Let $\widehat{T}_{\alpha,\beta}(t)$ and $\widehat{T}_{\alpha,\beta}^n(t)$ be the associated semigroups on $\Lb_{\alpha,\beta}$.
 By Trotter-Kato approximation and since $B_{bs}(\Gamma_0^2)$ is a core for the generators of both semigroups,
 it suffices to show $\widehat{L}_nG \longrightarrow \widehat{L}G$ for any $G \in B_{bs}(\Gamma_0^2)$.
 To this end one can show that
 \[
  \Vert \widehat{L}_nG - \widehat{L}G \Vert_{\Lb_{\alpha,\beta}} \leq \int _{\Gamma_0^2}c_n(\alpha,\beta;\eta)|G(\eta)|e^{\alpha|\eta^+|}e^{\beta|\eta^-|}\,d \lambda(\eta).
 \]
 Note that the integrand tends for a.a. $\eta$ to zero. Since $c_n(\alpha,\beta;\eta) \leq a(\alpha,\beta)(M(\eta) + M_n(\eta))$ and by
 $M(\eta) = \lim _{n \to \infty}\ M_n(\eta) \leq C|\eta|^{N+1}e^{\tau|\eta|}$ it follows that
 $c_n(\alpha,\beta;\eta) \leq 2 a(\alpha,\beta)C |\eta|^{N+1}e^{\tau |\eta|}$. Dominated convergence implies $\widehat{L}_nG \longrightarrow \widehat{L}G$, as $n\to \infty$
 for all $G \in B_{bs}(\Gamma_0^2)$.
\end{proof}

\subsection{Correlation functions}
Suppose that condition (A) and \eqref{PHDIPS:02} are fulfilled.
Denote by $\widehat{T}_{\alpha,\beta}(t)^*$ the adjoint semigroup on $\K_{\alpha,\beta}$ and by $(\widehat{L}^*, D_{\alpha,\beta}(\widehat{L}^*))$
the adjoint operator to $(\widehat{L}, D_{\alpha,\beta}(\widehat{L}))$, i.e. $\langle \widehat{L}G, k\rangle = \langle G, \widehat{L}^*k\rangle$
for $G \in D_{\alpha,\beta}(\widehat{L})$ and $k \in D_{\alpha,\beta}(\widehat{L}^*)$.
\begin{Remark}\label{GMCSRK:04}
 Let $\alpha' < \alpha$ and $\beta' < \beta$ be such that condition \eqref{PHDIPS:02} holds for $\alpha',\beta'$ and $\alpha,\beta$.
 Let $(\widehat{T}_{\alpha',\beta'}(s))_{s \geq 0}$ be the analytic semigroup constructed in Theorem \ref{PHDIPSTH:01}.
 Then by \eqref{GMCS:01} for any $G \in \Lb_{\alpha,\beta} \subset \Lb_{\alpha',\beta'}$ and $k \in \K_{\alpha',\beta'} \subset \K_{\alpha,\beta}$ we obtain
 \[
  \langle G, \widehat{T}_{\alpha',\beta'}(t)^*k \rangle = \langle \widehat{T}_{\alpha',\beta'}(t)G, k\rangle = \langle \widehat{T}_{\alpha,\beta}(t)G, k\rangle = \langle G, \widehat{T}_{\alpha,\beta}(t)^*k\rangle
 \]
 and hence $\widehat{T}_{\alpha,\beta}(t)^*k = \widehat{T}_{\alpha',\beta'}(t)^*k$ holds.
\end{Remark}
We consider the linear operator $L^{\Delta}$
\begin{align*}
 (L^{\Delta}k)(\eta) = &- \sum _{x \in \eta^-}\int _{\Gamma_0^2}k(\eta \cup \xi)
 (\mathbb{K}_0^{-1}d^-(x,\cdot \cup \eta^+, \cdot \cup \eta^- \backslash x))(\xi)\,d\lambda(\xi)
  \\ &- \sum _{x \in \eta^+}\int _{\Gamma_0^2}k(\eta \cup \xi)(\mathbb{K}_0^{-1}d^+(x,\cdot \cup
  \eta^+ \backslash x, \cdot \cup \eta^-))(\xi)\,d\lambda(\xi)
  \\ &+ \sum _{x \in \eta^-}\int _{\Gamma_0^2}k(\eta^+ \cup \xi^+, \eta^- \cup \xi^- \backslash x)
  (\mathbb{K}_0^{-1}b^-(x,\cdot \cup \eta^+, \cdot \cup \eta^-\backslash x))(\xi)\,d\lambda(\xi)
  \\ &+ \sum _{x \in \eta^+}\int _{\Gamma_0^2}k(\eta^+ \cup \xi^+ \backslash x, \eta^- \cup \xi^-)
  (\mathbb{K}_0^{-1}b^+(x,\cdot \cup \eta^+ \backslash x, \cdot \cup \eta^-))(\xi)\,d\lambda(\xi)
\end{align*}
on the maximal domain $D_{\alpha,\beta}(L^{\Delta}) = \{ k \in \K_{\alpha,\beta}\ | \ L^{\Delta}k \in \K_{\alpha,\beta} \}$.
Introduce the following additional conditions.
\begin{enumerate}
 \item[(B)] There exist constants $C > 0$, $\tau \geq 0$ and $N \in \N$ such that for all $x \in \R^d$ and $\eta \in \Gamma_0^2$
 \begin{align}\label{PHDIPS:25}
  d^+(x,\eta) + d^-(x,\eta) + b^+(x,\eta) + b^-(x,\eta) \leq C(1+|\eta|)^Ne^{\tau|\eta|}
 \end{align}
 holds.
 \item[(C)] Let $\alpha,\beta$ and $a(\alpha,\beta)$ be as in \eqref{PHDIPS:02}.
 There exist constants $\alpha',\beta' \in \R$ with $\alpha' + \tau < \alpha$, $\beta' + \tau < \beta$ and another constant $a(\alpha',\beta') > 0$ such that
 \[
  c(\alpha',\beta';\eta) \leq a(\alpha',\beta')M(\eta), \quad \eta \in \Gamma_0^2
 \]
 holds.
\end{enumerate}
\begin{Lemma}\label{GMCSLEMMA:00}
 Suppose that \eqref{PHDIPS:02} and conditions (A)--(C) are satisfied. Then
 \[
  \widehat{L}^* = L^{\Delta} \in L(\K_{\alpha',\beta'}, \K_{\alpha,\beta}).
 \]
 In particular $\K_{\alpha',\beta'} \subset D_{\alpha,\beta}(L^{\Delta})$ holds.
\end{Lemma}
\begin{proof}
 Observe that for $k \in \K_{\alpha',\beta'}$
 \begin{align*}
  |L^{\Delta}k(\eta)| \leq \Vert k \Vert_{\K_{\alpha',\beta'}}a(\alpha',\beta') C |\eta|^{N+1} e^{(\alpha' + \tau)|\eta^+|}e^{(\beta' + \tau)|\eta^-|}.
 \end{align*}
 Then
 \begin{align*}
  |\eta|^{N+1} e^{-(\alpha - \alpha' - \tau)|\eta^+|}e^{-(\beta - \beta' - \tau)|\eta^-|}
  &\leq |\eta|^{N+1}e^{-\min\{\alpha - \alpha' - \tau, \beta - \beta' - \tau\}|\eta|}
  \\ &\leq \frac{(N+1)^{N+1}e^{-(N+1)}}{(\min\{\alpha - \alpha' - \tau, \beta - \beta' - \tau\})^{N+1}}
 \end{align*}
 implies
 \[
  |L^{\Delta}k(\eta)| \leq \frac{C a(\alpha',\beta')(N+1)^{N+1}e^{-(N+1)}}{(\min\{\alpha - \alpha' - \tau, \beta - \beta' - \tau\})^{N+1}}e^{\alpha|\eta^+|}e^{\beta|\eta^-|}\Vert k \Vert_{\K_{\alpha',\beta'}},
 \]
 i.e. $L^{\Delta}$ is bounded from $\K_{\alpha',\beta'}$ to $\K_{\alpha,\beta}$.
 Let us prove the second assertion. Define a positive linear mapping
 \begin{align*}
 (\mathcal{C}k)(\eta) &= \sum _{x \in \eta^-}\int _{\Gamma_0^2}k(\eta \cup \xi)
 |\mathbb{K}_0^{-1}d^-(x,\cdot \cup \eta^+, \cdot \cup \eta^- \backslash x)|(\xi)\,d\lambda(\xi)
  \\ &\ \ + \sum _{x \in \eta^+}\int _{\Gamma_0^2}k(\eta \cup \xi)|\mathbb{K}_0^{-1}d^+(x,\cdot \cup
  \eta^+ \backslash x, \cdot \cup \eta^-)|(\xi)\,d\lambda(\xi)
  \\ &\ \ + \sum _{x \in \eta^-}\int _{\Gamma_0^2}k(\eta^+ \cup \xi^+, \eta^- \cup \xi^- \backslash x)
  |\mathbb{K}_0^{-1}b^-(x,\cdot \cup \eta^+, \cdot \cup \eta^-\backslash x)|(\xi)\,d\lambda(\xi)
  \\ &\ \ + \sum _{x \in \eta^+}\int _{\Gamma_0^2}k(\eta^+ \cup \xi^+ \backslash x, \eta^- \cup \xi^-)
  |\mathbb{K}_0^{-1}b^+(x,\cdot \cup \eta^+ \backslash x, \cdot \cup \eta^-)|(\xi)\,d\lambda(\xi).
 \end{align*}
 Then, for $G \in B_{bs}(\Gamma_0^2)$ and $k \in \K_{\alpha,\beta}$ we have
 \[
  |k(\eta)|\cdot (M(\eta)|G(\eta)| + B'|G|(\eta)) \leq \Vert k \Vert_{\K_{\alpha,\beta}}e^{\alpha|\eta^+|}e^{\beta|\eta^-|}(M(\eta)|G(\eta)| + B'|G|(\eta))
 \]
 and
 %\[
 $$
  |G(\eta)|\cdot \mathcal{C}|k|(\eta) \leq c(\alpha,\beta;\eta)\Vert k \Vert_{\K_{\alpha,\beta}}
  |G(\eta)|e^{\alpha|\eta^+|}e^{\beta|\eta^-|}
 $$
 %\]
 which shows that $|k|\cdot (A + B')|G|$ and $|G(\eta)|\cdot \mathcal{C}|k|$ are both integrable. Hence we can apply \eqref{PHD:01} and
 obtain by a simple computation
 \[
  \int _{\Gamma_0^2}(\widehat{L}G)(\eta)k(\eta)\,d\lambda(\eta) = \int _{\Gamma_0^2}G(\eta)(L^{\Delta}k)(\eta)\,d\lambda(\eta).
 \]
 Let $k \in D_{\alpha,\beta}(L^{\Delta})$, then above equality implies $L^{\Delta} \subset \widehat{L}^*$.
 Conversely take $k \in D_{\alpha,\beta}(\widehat{L}^*)$. Then
 \begin{align*}
  \int _{\Gamma_0^2}G(\eta)(\widehat{L}^*k)(\eta)\,d\lambda(\eta) = \int _{\Gamma_0^2}(\widehat{L}G)(\eta)k(\eta)\,d\lambda(\eta)
  = \int _{\Gamma_0^2}G(\eta)(L^{\Delta}k)(\eta)\,d\lambda(\eta)
 \end{align*}
 implies $L^{\Delta}k = \widehat{L}^*k \in \K_{\alpha,\beta}$ and $D_{\alpha,\beta}(\widehat{L}^*) \subset D_{\alpha,\beta}(L^{\Delta})$.
\end{proof}

Since $\Lb_{\alpha,\beta}$ is not reflexive, $\widehat{T}_{\alpha,\beta}(t)^*$ does not need to be strongly continuous.
However, it is continuous w.r.t. the topology $\sigma((\Lb_{\alpha,\beta})^*, \K_{\alpha,\beta}) = \sigma( \K_{\alpha,\beta}, \Lb_{\alpha,\beta})$.
Here $\sigma(\K_{\alpha,\beta}, \Lb_{\alpha,\beta})$ is the smallest topology such that for each $k \in \K_{\alpha,\beta}$,
$\Lb_{\alpha,\beta} \ni G \longmapsto \langle G, k \rangle$ is continuous.
It is well-known that $\widehat{T}_{\alpha,\beta}(t)^*$ is strongly continuous on the proper subspace $\K_{\alpha,\beta}^{\odot} = \overline{D_{\alpha,\beta}(L^{\Delta})}$
and its restriction $\widehat{T}_{\alpha,\beta}(t)^{\odot} := \widehat{T}_{\alpha,\beta}(t)^*|_{\K_{\alpha,\beta}^{\odot}}$ is a $C_0$-semigroup with generator
$\widehat{L}^{\odot}k = L^{\Delta}k$,
\[
 D_{\alpha,\beta}(\widehat{L}^{\odot}) = \{ k \in D_{\alpha,\beta}(L^{\Delta})\ | \ L^{\Delta}k \in \K_{\alpha,\beta}^{\odot}\},
\]
see e.g. \cite{ENG00}. Hence, for each $k_0 \in D_{\alpha,\beta}(\widehat{L}^{\odot})$ there exists a unique classical solution to \eqref{PHDIPS:11}
on the Banach space $\mathcal{K}_{\alpha,\beta}^{\odot}$, which is given by $\widehat{T}_{\alpha,\beta}(t)^*k_0 =: k_t$
(that is $k_t$ is continuously differentiable w.r.t. the norm in $\K_{\alpha,\beta}$ and its derivative satisfies \eqref{PHDIPS:11}).
\begin{Remark}
 Note that $D_{\alpha,\beta}(\widehat{L}^{\odot})$ depends on the generator. Using the theory of sun-dual semigroups it is possible to find certain invariant subspaces
 which are only characterized by the constant $\tau$ in condition (B), and prove that
 the restriction of $\widehat{T}_{\alpha,\beta}(t)^*$ onto such spaces is a strongly continuous semigroup.
 This can be used to prove existence and uniqueness of classical solutions to \eqref{PHDIPS:11} on this subspaces (see \cite{FKK12, FKK15}).
\end{Remark}
Below we consider another approach, which can be used to show weak uniqueness of solutions to \eqref{PHDIPS:11}.
Let $\mathcal{C}(\K_{\alpha,\beta}, \Lb_{\alpha,\beta}) =: \mathcal{C}$ be the topology of
uniform convergence on compact subsets of $\Lb_{\alpha,\beta}$. A basis of neighbourhoods around $0$ is given by sets of the form
\[
 \{ k \in \K_{\alpha,\beta}\ | \ \sup _{G \in K}|\langle G, k\rangle| < \e \},
\]
with $\e > 0$ and a compact $K \subset \Lb_{\alpha,\beta}$, see \cite{L10, WZ02, WZ06} and the references therein.
The semigroup $(\widehat{T}_{\alpha,\beta}(t)^*)_{t \geq 0}$ becomes continuous w.r.t. $\mathcal{C}$ and its generator w.r.t. $\mathcal{C}$ is exactly the adjoint
operator $(L^{\Delta}, D_{\alpha,\beta}(L^{\Delta}))$, cf. \cite[Theorem 1.4]{WZ06}.
The next theorem provides existence, uniqueness and regularity of solutions to the Cauchy problem
\begin{align}\label{INTRO:02}
 \frac{d}{dt}\langle G, k_t \rangle = \langle \widehat{L}G, k_t \rangle, \quad k_t|_{t=0} = k_0,
 \quad G \in B_{bs}(\Gamma_0^2).
\end{align}
\begin{Theorem}\label{GMCSTH:01}
 Suppose that \eqref{PHDIPS:02} and (A) are satisfied. Then for any $k_0 \in \K_{\alpha,\beta}$ the equation \eqref{INTRO:02}
 has a unique solution given by $k_t = \widehat{T}_{\alpha,\beta}(t)^*k_0$. This means that $k_t$ is continuous w.r.t. the topology $\mathcal{C}$ and satisfies
 \begin{align}\label{GMCS:23}
  \langle G, k_t\rangle = \langle G, k_0 \rangle + \int _{0}^{t}\langle \widehat{L}G, k_s \rangle \,ds, \quad G \in B_{bs}(\Gamma_0^2).
 \end{align}
 Assume that conditions (B) and (C) are fulfilled and let $\alpha',\beta'$ be the corresponding constants. Then the following assertions are true:
 \begin{enumerate}
  \item If $k_0 \in \K_{\alpha', \beta'}$, then $k_t$ is continuous w.r.t. to the norm in $\K_{\alpha,\beta}$.
  \item If $k_0 \in \K_{\alpha',\beta'}$ and $\alpha' + 2\tau < \alpha$, $\beta' + 2 \tau < \beta$, then $k_t$ is also continuously differentiable w.r.t. to the norm
  in $\K_{\alpha,\beta}$ and the unique classical solution to \eqref{PHDIPS:11}.
 \end{enumerate}
\end{Theorem}
\begin{proof}
 Existence and uniqueness for the Cauchy problem \eqref{GMCS:23} follows from \cite[Theorem~2.1]{WZ06} and Theorem \ref{PHDIPSTH:01} if one replaces
 $B_{bs}(\Gamma_0^2)$ by $D_{\alpha,\beta}(\widehat{L})$ in \eqref{GMCS:23}. But since $B_{bs}(\Gamma_0^2)$ is a core, uniqueness is proved.
 Moreover, since $\widehat{T}_{\alpha,\beta}(t)^*$ is continuous w.r.t. $\sigma(\K_{\alpha,\beta}, \Lb_{\alpha,\beta})$,
 $t \longmapsto \langle \widehat{L}G, k_t\rangle$ is continuous and hence \eqref{GMCS:23} implies \eqref{INTRO:02}.
 %\\

 1. If $k_0 \in \K_{\alpha',\beta'}$, then by Lemma \ref{GMCSLEMMA:00} $L^{\Delta}k_0 \in \K_{\alpha,\beta}$ and hence
 $k_0 \in D_{\alpha,\beta}(L^{\Delta}) \subset \K_{\alpha,\beta}^{\odot}$ which implies the assertion.
 %\\

 2. Suppose that $\alpha' + 2\tau < \alpha$, $\beta' + 2\tau < \beta$ and let $\alpha'' \in (\alpha', \alpha)$, $\beta'' \in (\beta', \beta)$ be such that
 $\alpha' + \tau < \alpha''$, $\alpha'' + \tau < \alpha$ and $\beta' + \tau < \beta''$, $\beta'' + \tau < \beta$.
 By Lemma \ref{GMCSLEMMA:00} the operator $L^{\Delta}$ is bounded as $\K_{\alpha',\beta'} \to \K_{\alpha'',\beta''}$ and $\K_{\alpha'',\beta''} \to \K_{\alpha,\beta}$.
 Therefore $k_0 \in D_{\alpha,\beta}(L^{\Delta})$ and $L^{\Delta}k_0 \in \K_{\alpha'',\beta''} \subset D_{\alpha,\beta}(L^{\Delta})$.
 Thus $k_0 \in D_{\alpha,\beta}(\widehat{L}^{\odot})$ implies that $k_t$
 is continuously differentiable w.r.t. the norm in $\K_{\alpha,\beta}$ and it is a classical solution to \eqref{PHDIPS:11}.
\end{proof}

\subsection{Positive definiteness}
Suppose that (A) and \eqref{PHDIPS:02} are satisfied.
We start with the definition of solutions to the Fokker-Planck equation \eqref{INTRO:01}.
\begin{Definition}
 A family of Borel probability measures $(\mu_t)_{t \geq 0} \subset \mathcal{P}_{\alpha,\beta}$ is said to be a weak solution to \eqref{INTRO:01}
 if for any $F \in \mathcal{FP}(\Gamma^2)$, $t \longmapsto \langle LF, \mu_t\rangle$ is locally integrable and satisfies
 \begin{align}\label{GMCS:26}
  \langle F,\mu_t \rangle = \langle F, \mu_0 \rangle + \int _{0}^{t} \langle LF, \mu_s \rangle \,ds, \quad t \geq 0.
 \end{align}
\end{Definition}
Note that $\langle F, \mu_t\rangle$ is well-defined since $F$ is a polynomially bounded cylinder function and $\mu_t$ has finite local moment.
Uniqueness is stated in the next theorem.
\begin{Theorem}\label{GMCSLEMMA:04}(Uniqueness)\
 Suppose that (A) and \eqref{PHDIPS:02} are fulfilled. Then equation \eqref{INTRO:01} has at most one weak solution $(\mu_t)_{t \geq 0} \subset \mathcal{P}_{\alpha,\beta}$
 such that its correlation functions $(k_t)_{t \geq 0}$ satisfy
 \[
  \sup _{t \in [0,T]}\ \Vert k_t \Vert_{\K_{\alpha,\beta}} < \infty, \quad \forall T > 0.
 \]
\end{Theorem}
\begin{proof}
 Let $(\mu_t)_{t \geq 0} \subset \mathcal{P}_{\alpha,\beta}$ be a solution to \eqref{INTRO:01}, and denote by $(k_t)_{t \geq 0} \subset \K_{\alpha,\beta}$
 the associated correlation functions. Let $F \in \mathcal{FP}(\Gamma^2)$ and $G \in B_{bs}(\Gamma_0^2) \subset D_{\alpha,\beta}(\widehat{L})$ such that $F = \mathbb{K}G$.
 Then by $k_t(\eta) \leq \Vert k_t \Vert_{\K_{\alpha,\beta}}e^{\alpha|\eta^+|}e^{\beta|\eta^-|}$
 it follows that $G, \widehat{L}G \in \Lb_{\alpha,\beta} \subset \Lb_{k_t}$.
 Since $\mathbb{K}: \Lb_{k_t} \longrightarrow L^1(\Gamma^2, d\mu_t)$
 is continuous it follows that $F = \mathbb{K}G, L\mathbb{K}G = \mathbb{K} \widehat{L}G$ belong to $L^1(\Gamma^2, d\mu_t)$ for any $t \geq 0$.
 Moreover,
 \[
  \langle LF, \mu_t\rangle = \langle \mathbb{K}\widehat{L}G, \mu_t\rangle = \langle \widehat{L}G, k_t\rangle
 \]
 shows that $t \longmapsto \langle \widehat{L}G, k_t \rangle$ is locally integrable and hence
 \[
  \langle G, k_t\rangle = \langle G, k_0\rangle + \int _{0}^{t}\langle \widehat{L}G, k_s \rangle \,ds,
  \quad t \geq 0, \quad G \in B_{bs}(\Gamma_0^2)
 \]
 holds. Thus $k_t$ is continuous w.r.t. $\sigma(\K_{\alpha,\beta}, \Lb_{\alpha,\beta})$ and since $k_t$ is norm-bounded on $[0,T]$ \cite[Lemma 1.10]{WZ06}
 implies that $k_t$ is also continuous w.r.t. the topology $\mathcal{C}$.
 As a consequence $(k_t)_{t \geq 0}$ is a weak solution to \eqref{GMCS:23} and hence it is given by $k_t = \widehat{T}_{\alpha,\beta}(t)^*k_0$.
\end{proof}
\begin{Remark}
 Let $k_0 \in \K_{\alpha,\beta}$ be positive definite and suppose that $k_t := \widehat{T}_{\alpha,\beta}(t)^*k_0 \in \K_{\alpha,\beta}$ is positive definite.
 Then $(k_t)_{t \geq 0}$ is a weak solution to \eqref{GMCS:23} and for each $t \geq 0$ there exists a unique $\mu_t \in \mathcal{P}_{\alpha,\beta}$
 having correlation function $k_t$. By $\langle G,k_t \rangle = \langle F, \mu_t \rangle$ and
 $\langle \widehat{L}G, k_t\rangle = \langle LF, \mu_t \rangle$ it follows that $(\mu_t)_{t \geq 0}$ is a weak solution to \eqref{INTRO:01}.
\end{Remark}
Above considerations show that for existence of weak solutions to \eqref{INTRO:01}, it suffices to show that $\widehat{T}_{\alpha,\beta}(t)^*$
preserves the cone of positive definite functions.
To this and we approximate $k_t = \widehat{T}_{\alpha,\beta}(t)^*k_0$ by an auxiliary evolution $\widehat{T}_{\alpha,\beta}^{\delta}(t)^*k_0$ and prove that
$\widehat{T}_{\alpha,\beta}^{\delta}(t)^*k_0$ is positive definite.
Such approximation scheme was used for particular models in \cite{FKKZ12, KK16}.

Let $(R_{\delta})_{\delta > 0}$ be a sequence of continuous
integrable functions with $0 < R_{\delta} \leq 1$ and
$R_{\delta}(x)
%\nnearrow
\nearrow 1$ as $\delta \to 0$ for all $x \in \R^d$. In the
following we simply say that $(R_{\delta})_{\delta > 0}$ is a
localization sequence. Define new birth intensities by
$b^+_{\delta}(x,\eta) := R_{\delta}(x)b^+(x,\eta)$ and
$b^-_{\delta}(x,\eta) := R_{\delta}(x)b^-(x,\eta)$ for all $x \in
\R^d$ and $\eta \in \Gamma_0^2$. Then for any $\eta \in
\Gamma_0^2$ and $\delta > 0$
\begin{align}\label{GMCS:29}
 \int _{\R^d}\left(b_{\delta}^+(x,\eta) + b_{\delta}^-(x,\eta)\right)\,dx < \infty
\end{align}
holds. Denote by $L_{\delta}$ the operator $L$ with $b^+, b^-$ replaced by $b_{\delta}^+, b_{\delta}^-$.
We will consider the latter operator on a proper set of functions $F: \Gamma_0^2 \longrightarrow \R$.

Let $D_{\delta}(\eta) = M(\eta) + \int
_{\R^d}b_{\delta}^+(x,\eta)\,dx + \int
_{\R^d}b_{\delta}^-(x,\eta)\,dx$ and define
\begin{align*}
 Q_{\delta}\rho(\eta) &= \int _{\R^d}d^-(x,\eta)\rho(\eta^+,\eta^- \cup x)\,dx + \int _{\R^d}d^+(x,\eta)\rho(\eta^+
 \cup x,\eta^-)\,dx
 \\ & \ \ \ + \sum _{x \in \eta^-}b_{\delta}^-(x,\eta^+,\eta^- \backslash x)\rho(\eta^+,\eta^- \backslash x) + \sum _{x \in \eta^+}b_{\delta}^+(x,\eta^+ \backslash x, \eta^-)\rho(\eta^+ \backslash x, \eta^-).
\end{align*}
The linear operator $\mathcal{I}_{\delta} = -D_{\delta} + Q_{\delta}$ is considered on the domain
\[
 D(\mathcal{I}_{\delta}) = \{ \rho \in L^1(\Gamma_0^2, \,d\lambda) \ | \ D_{\delta} \rho \in L^1(\Gamma_0^2, \,d\lambda)\},
\]
where $D_{\delta}$ acts as a multiplication operator on $L^1(\Gamma_0^2,\, d\lambda)$.
Note that an analogue of $D_{\delta}$ was already introduced in \cite{P75}.
Since $Q_{\delta}$ preserves positivity and satisfies
\[
 \int_{\Gamma_0^2}Q_{\delta}\rho(\eta)\,d\lambda(\eta) = \int_{\Gamma_0^2}D_{\delta}(\eta)\rho(\eta)
 \,d\lambda(\eta), \quad 0 \leq  \rho \in D(\mathcal{I}_{\delta}),
\]
it follows that
$(\mathcal{I}_{\delta}, D(\mathcal{I}_{\delta}))$ has an extension $(\mathcal{G}_{\delta}, D(\mathcal{G}_{\delta}))$
which is the generator of a sub-stochastic semigroup
$(S_{\delta}(t))_{t \geq 0}$ on $L^1(\Gamma_0^2, d\lambda)$, cf. \cite[Theorem 2.2]{TV06}.
Here sub-stochastic means that $S_{\delta}(t)$ is a strongly continuous semigroup of contractions which preserves positivity.
Moreover, this semigroup is minimal in the sense that given another sub-stochastic semigroup $U_{\delta}(t)$
such that its generator is an extension of $(\mathcal{I}_{\delta}, D(\mathcal{I}_{\delta}))$, then $S_{\delta}(t) \leq U_{\delta}(t)$.
Note that, in general, $\Vert S_{\delta}(t)\rho_0 \Vert_{L^1(\Gamma_0^2,\,d\lambda)} < \Vert \rho_0 \Vert_{L^1(\Gamma_0^2,\,d\lambda)}$
may happen for some $t > 0$ (see e.g. \cite{BA06} and the references therein).
It is related to $\overline{\mathcal{I}}_{\delta} = \mathcal{G}_{\delta}$ and is the main ingredient for the next condition.
\begin{enumerate}
 \item[(D)] There exists a localization sequence $(R_{\delta})_{\delta > 0}$ such that
 $\overline{\mathcal{I}}_{\delta} = \mathcal{G}_{\delta}$, i.e. for each $\rho_0 \in D(\mathcal{G}_{\delta})$ there exists a unique classical solution to
 \begin{align}\label{GMCS:31}
  \frac{\partial \rho^{\delta}_t}{\partial t} = \mathcal{G}_{\delta}\rho^{\delta}_t, \quad \rho^{\delta}_t|_{t=0} = \rho_0.
 \end{align}
\end{enumerate}
Note that condition (D) implies that the semigroup $S_{\delta}(t)$ is stochastic.
For technical reasons we will need the adjoint semigroup on $L^{\infty}(\Gamma_0^2,\,d\lambda)$.
Let $(\mathcal{J}_{\delta},D(\mathcal{J}_{\delta}))$ be the adjoint operator to
$(\mathcal{I}_{\delta}, D(\mathcal{I}_{\delta}))$ on
$L^{\infty}(\Gamma_0^2,\,d\lambda)$. The next lemma follows by standard arguments and \eqref{PHD:01}.
\begin{Lemma}\label{GMCSLEMMA:03}
 For any $F \in D(\mathcal{J}_{\delta})$ it holds that $\mathcal{J}_{\delta}F = L_{\delta}F$.
\end{Lemma}
\begin{Remark}
 Condition (D) is a non-explosion condition. It is fulfilled, provided one can find a proper Lyapunov functional.
 Sufficient conditions and related results how such condition can be checked are given in \cite{F16, FK16}.
 A general (and rather easy) criterion can be found in \cite{TV06}. A general approach to study condition (D) in an abstract
 setting is given in \cite{BA06} (see also the references therein).
\end{Remark}
The next statement provides existence and uniqueness of solutions to \eqref{INTRO:01}.
\begin{Proposition}\label{PHDIPSTH:02}(Existence)
 Suppose that (A)--(D) and \eqref{PHDIPS:02} are fulfilled.
 Then $\widehat{T}_{\alpha,\beta}(t)^* \K_{\alpha', \beta'}^+ \subset \K_{\alpha,\beta}^+$. In particular for any $\mu_0 \in \mathcal{P}_{\alpha',\beta'}$
 there exists exactly one weak solution $(\mu_t)_{t \geq 0} \subset \mathcal{P}_{\alpha,\beta}$ to \eqref{INTRO:01}.
 The correlation functions are given by $k_{\mu_t} = \widehat{T}_{\alpha,\beta}(t)^*k_{\mu_0}$.
 If conditions (B) and (C) hold for all $\tau > 0$, then $\widehat{T}_{\alpha,\beta}(t)^* \K_{\alpha, \beta}^+ \subset \K_{\alpha,\beta}^+$ holds.
\end{Proposition}
Existence of an associated Markov function is stated in the next corollary, cf. \cite{KKM08}.
\begin{Corollary}
 Suppose that (A)--(D) hold for any $\tau > 0$ and assume that \eqref{PHDIPS:02} holds.
 Then for any $\mu \in \mathcal{P}_{\alpha,\beta}$ there exists a Markov function $(X_t^{\mu})_{t \geq 0}$
 on the configuration space $\Gamma^2$ with the initial distribution $\mu$ associated with the generator $L$.
\end{Corollary}
The rest of this section is devoted to the proof of Proposition \ref{PHDIPSTH:02}.
\begin{Lemma}\label{GMCSLEMMA:05}
 For any $\delta > 0$ Theorem \ref{PHDIPSTH:01} and Theorem \ref{GMCSTH:01} hold with $L$ replaced by $L_{\delta}$.
 Let $\widehat{T}^{\delta}_{\alpha,\beta}(t)$ be the semigroup on $\Lb_{\alpha,\beta}$.
 Then, for any $G \in \Lb_{\alpha,\beta}$
 \[
  \widehat{T}^{\delta}_{\alpha,\beta}(t)G \longrightarrow \widehat{T}_{\alpha,\beta}(t)G, \quad \delta \to 0
 \]
 holds uniformly on compacts in $t \geq 0$.
\end{Lemma}
\begin{proof}
 The first claim follows by $R_{\delta}(x) \leq 1$ and an repetition of the arguments given in the proofs of Theorem \ref{PHDIPSTH:01} and Theorem \ref{GMCSTH:01}.
 In particular, such repetition shows that $\widehat{L}_{\delta} := \mathbb{K}_0^{-1}L_{\delta}\mathbb{K}_0$ is given by
 $\widehat{L}_{\delta} = A + B_{\delta}$, where $A$ is given as before and $B_{\delta}$ is obtained from $B$
 by multiplication of the terms for the birth by $R_{\delta}(x)$.
 In particular, it is the generator of an analytic semigroup $\widehat{T}_{\alpha,\beta}^{\delta}(t)$ of contractions
 on $\Lb_{\alpha,\beta}$. Note that this generator is considered on the same domain $D_{\alpha,\beta}(\widehat{L})$ as $\widehat{L}$.
 It is not difficult to see that $\widehat{L}_{\delta}G \longrightarrow \widehat{L}G$, as $\delta \to 0$ for all $G \in D_{\alpha,\beta}(\widehat{L})$.
 In view of Trotter-Kato approximation the assertion is proved.
\end{proof}

Let $\mathcal{B}_{\alpha,\beta}$ be the Banach space of all
equivalence classes of functions $G$ with norm
\[
 %\vvvert
 {\mid \parallel }
 G
 %\vvvert
{\mid \parallel }_
 {\mathcal{B}_{\alpha,\beta}} =
 \int _{\Gamma_0^2}|G(\eta)|e_{\lambda}(R_{\delta}; \eta^+)e_{\lambda}(R_{\delta};\eta^-)e^{\alpha|\eta^+
 |}e^{\beta|\eta^-|}
 \,d\lambda(\eta).
\]
Its dual Banach space $\mathcal{B}_{\alpha,\beta}^*$ can be identified with the Banach space of all equivalence classes of functions $k$ with norm
\[
 %\vvvert
 {\mid \parallel }
 k
 %\vvvert
  {\mid \parallel } _{\mathcal{B}_{\alpha,\beta}^*} = \esssup _{\eta \in \Gamma_0^2}\frac{|k(\eta)|}{e_{\lambda}(R_{\delta};\eta^+)e_{\lambda}(R_{\delta};\eta^-)e^{\alpha|\eta^+|}e^{\beta|\eta^-|}}.
\]
Here and in the following we let $\delta > 0$ be arbitrary but fixed. To omit cumbersome notation will not explicitly state
the dependence of $\mathcal{B}_{\alpha,\beta}, \mathcal{B}_{\alpha,\beta}^*$ on $\delta$.
\begin{Remark}
 Note that such spaces have the following two important properties.
 \begin{enumerate}
  \item[1.] The space $\mathcal{B}_{\alpha,\beta}$ is large in the sense that
  \[
   \bigcup _{\alpha'',\beta'' \in \mathbb{R}} \mathcal{K}_{\alpha'',\beta''} \subset \mathcal{B}_{\alpha,\beta}.
  \]
  \item[2.] The space $\mathcal{B}_{\alpha,\beta}^*$ is small in the sense that
  \[
   \mathcal{B}_{\alpha,\beta}^* \subset \bigcap _{\alpha'',\beta'' \in \mathbb{R}}\Lb_{\alpha'',\beta''}.
  \]
 \end{enumerate}
\end{Remark}
The same arguments as for the proof of Theorem \ref{PHDIPSTH:01}
and Theorem \ref{GMCSTH:01} show that we can replace $\Lb_{\alpha,\beta},
\K_{\alpha,\beta}$ also by $\mathcal{B}_{\alpha,\beta}$ and
$\mathcal{B}_{\alpha,\beta}^*$. Note that it requires to introduce another function $c_{\delta}(\alpha,\beta;\eta)$
which is an analogue of $c(\alpha,\beta;\eta)$.
A simple computation shows that it is given by $c(\alpha,\beta;\eta)$ with $e^{\alpha|\xi^+|}e^{\beta|\xi^-|}$ replaced by
$e_{\lambda}(R_{\delta};\xi^+)e_{\lambda}(R_{\delta};\xi^-)e^{\alpha|\xi^+|}e^{\beta|\xi^-|}$.
At this point it is necessary to use the additional factor $R_{\delta}$ in $b_{\delta}^{\pm}$.
Now $c_{\delta}(\alpha,\beta;\eta) \leq c(\alpha,\beta;\eta)$ shows that the same arguments as for the proof of Theorem \ref{PHDIPSTH:01}
and Theorem \ref{GMCSTH:01} can be applied. Denote by $U_{\delta}(t)$ and
$U_{\delta}(t)^*$ the corresponding semigroups on
$\mathcal{B}_{\alpha,\beta}$ and $\mathcal{B}_{\alpha,\beta}^*$,
respectively. Let $(\widehat{L}_{\delta},
D^{\mathcal{B}}_{\alpha,\beta}(\widehat{L}))$ be the generator of
$U_{\delta}(t)$. The proofs of Theorem \ref{PHDIPSTH:01} and
\ref{GMCSTH:01} show that
\[
 D^{\mathcal{B}}_{\alpha,\beta}(\widehat{L}) = \{ G \in \mathcal{B}_{\alpha,\beta}\ | \ M \cdot G \in \mathcal{B}_{\alpha,\beta}\}.
\]
Thus the Cauchy problem
\begin{align}\label{GMCS:28}
 \frac{d }{dt}\langle G, u^{\delta}_t \rangle = \langle \widehat{L}_{\delta}G, u^{\delta}_t \rangle,
 \quad  u^{\delta}_t|_{t=0} = u_0, \quad \forall G \in B_{bs}(\Gamma_0^2)
\end{align}
has for every $u_0 \in \mathcal{B}_{\alpha,\beta}^*$ a unique weak solution in $\mathcal{B}_{\alpha,\beta}^*$ given by $U_{\delta}(t)^*u_0$.
\begin{Lemma}\label{GMCSLEMMA:01}
 Let $k_0 \in \mathcal{B}_{\alpha,\beta}^*$, then $\widehat{T}_{\alpha,\beta}^{\delta}(t)^*k_0 = U_{\delta}(t)^*k_0$ holds.
\end{Lemma}
\begin{proof}
 First observe that $\mathcal{B}_{\alpha,\beta}^* \subset \K_{\alpha,\beta}$ continuously and hence $k_0 \in \K_{\alpha,\beta}$.
 In particular $u_t^{\delta} := U_{\delta}(t)^*k_0$ and $k_t^{\delta} := \widehat{T}_{\alpha,\beta}^{\delta}(t)^*k_0$ are well-defined.
 Moreover, since $\Lb_{\alpha,\beta}$ is continuously embedded into $\mathcal{B}_{\alpha,\beta}$ we obtain $D_{\alpha,\beta}(\widehat{L}) \subset D^{\mathcal{B}}_{\alpha,\beta}(\widehat{L})$,
 i.e. $(\widehat{L}_{\delta}, D^{\mathcal{B}}_{\alpha,\beta}(\widehat{L}))$ is an extension of $(\widehat{L}_{\delta}, D_{\alpha,\beta}(\widehat{L}))$.
 Therefore $(u^{\delta}_t)_{t \geq 0}$ is also a weak solution to \eqref{GMCS:23} and thus by uniqueness $u^{\delta}_t = k^{\delta}_t$, $t \geq 0$.
\end{proof}

\begin{Lemma}\label{GMCSLEMMA:02}
 Let $k_0 \in \mathcal{B}_{\alpha',\beta'}^*$ be positive definite. Denote by
 $u^{\delta}_t \in \mathcal{B}_{\alpha,\beta}^*$ the unique weak solution to \eqref{GMCS:28}, then $u^{\delta}_t$ is positive definite for any $t \geq 0$.
\end{Lemma}
\begin{proof}
 Define for any $u \in \mathcal{B}_{\alpha,\beta}^*$ a linear operator $\mathcal{H}u(\eta) := \int _{\Gamma_0^2}(-1)^{|\xi|}
 u(\eta \cup \xi)\,d\lambda(\xi)$.
 Then $\mathcal{H}u$ is well-defined and satisfies for any $C_+, C_- > 0$
 \begin{align*}
  & \int _{\Gamma_0^2} |\mathcal{H}u(\eta)|C_+^{|\eta^+|}C_-^{|\eta^-|}
  \,d\lambda(\eta)
  \\ &\leq
 % \vvvert
   {\mid \parallel }
  u
  %\vvvert
   {\mid \parallel }
  _{\mathcal{B}_{\alpha,\beta}^*} \int _{\Gamma_0^2}(1+C_+)^{|\eta^+|} (1+ C_-)^{|\eta^-|}e^{\alpha|\eta^+|}e^{\beta|\eta^-|}e_{\lambda}(R_{\delta};\eta^+)e_{\lambda}
  (R_{\delta};\eta^-)\,d\lambda(\eta),
 \end{align*}
 i.e. $\mathcal{H}: \mathcal{B}_{\alpha,\beta}^* \longrightarrow \Lb_{\log(C_+), \log(C_-)}$ is continuous.
 Let $G \in \mathcal{B}_{\alpha,\beta}$ be arbitrary, then for any $u \in \mathcal{B}_{\alpha,\beta}^*$ we get by Fubini's theorem and \eqref{PHD:01}
 that
 \begin{align}\label{GMCS:32}
  \langle \mathbb{K}_0G, \mathcal{H}u \rangle = \langle G, u\rangle
 \end{align}
 holds. We can apply Fubini's theorem and \eqref{PHD:01} since
 \begin{align*}
  &\ \int _{\Gamma_0^2}\int _{\Gamma_0^2}\int _{\Gamma_0^2}|G(\xi)||u(\eta
  \cup \xi \cup \zeta)|\,d \lambda(\zeta)d\lambda(\xi)d\lambda(\eta)
  \\ &\leq \Vert u \Vert_{\mathcal{B}_{\alpha,\beta}^*}e^{2 e^{\alpha}\langle R_{\delta}\rangle } e^{2 e^{\beta}
  \langle R_{\delta} \rangle} \int _{\Gamma_0^2}|G(\xi)|e^{\alpha|\xi^+|} e^{\beta|\xi^-|} e_{\lambda}
  (R_{\delta};\xi^+)e_{\lambda}(R_{\delta};\xi^-)\,d \lambda(\xi)
 \end{align*}
 is satisfied, where $\langle R_{\delta} \rangle := \int _{\R^d}R_{\delta}(x)\,d x$.
 For the same $u$ and $G \in D^{\mathcal{B}}_{\alpha,\beta}(\widehat{L})$ we obtain by \eqref{GMCS:32} and $\mathbb{K}_{0}\widehat{L}_{\delta}G = L_{\delta}\mathbb{K}_0G$
 \begin{align}\label{GMCS:33}
  \langle \widehat{L}_{\delta}G, u \rangle = \langle \mathbb{K}_0\widehat{L}_{\delta}G, \mathcal{H}u \rangle = \langle L_{\delta}\mathbb{K}_0G, \mathcal{H}u\rangle.
 \end{align}
 Now let $U_{\delta}(t)^*k_0 = u^{\delta}_t \in \mathcal{B}_{\alpha,\beta}^*$,
 then
 \begin{align*}
  \langle G,u^{\delta}_t\rangle = \langle G, u_0 \rangle + \int _{0}^{t}\langle \widehat{L}_{\delta}G,
  u^{\delta}_s \rangle \,d s, \quad G \in D_{\alpha,\beta}^{\mathcal{B}}(\widehat{L}).
 \end{align*}
 Observe that condition (B) implies $\mathcal{K}_{\log(2), \log(2)} \subset D_{\alpha,\beta}^{\mathcal{B}}(\widehat{L})$.
 Hence by \eqref{GMCS:32} and \eqref{GMCS:33} it follows
 for $R^{\delta}_t := \mathcal{H}u^{\delta}_t \in L^1(\Gamma_0^2, d\lambda)$, $t \geq 0$ that
 \[
  \langle \mathbb{K}_0G, R^{\delta}_t \rangle = \langle \mathbb{K}_0 G, R_0 \rangle +
  \int _{0}^{t}\langle L_{\delta}\mathbb{K}_0 G, R^{\delta}_s\rangle \,d s, \quad G \in \K_{\log(2), \log(2)}
 \]
 holds.
 For any $F \in D(\mathcal{J}_{\delta}) \subset L^{\infty}(\Gamma_0,\,d\lambda)$ we get $|\mathbb{K}_0^{-1}F(\eta)| \leq \Vert F \Vert_{L^{\infty}}2^{|\eta|}$
 and hence $D(\mathcal{J}_{\delta}) \subset \mathbb{K}_0\K_{\log(2), \log(2)}$.
 Thus we can find $G \in \K_{\log(2),\log(2)}$ such that $\mathbb{K}_0G = F \in D(\mathcal{J}_{\delta})$. Lemma \ref{GMCSLEMMA:03} therefore implies
 \[
  \langle F, R^{\delta}_t \rangle = \langle F, R_0 \rangle +
  \int _{0}^{t}\langle \mathcal{J}_{\delta}F, R^{\delta}_s \rangle \,d s, \quad F \in D(\mathcal{J}_{\delta}).
 \]
 Since $k_0 \in \mathcal{B}_{\alpha',\beta'}^*$ we get by Theorem \ref{GMCSTH:01}.1 that $u_t^{\delta}$ is continuous in $t \geq 0$
 w.r.t. the norm in $\mathcal{B}_{\alpha,\beta}^*$.
 Because $\mathcal{H}: \mathcal{B}_{\alpha,\beta}^* \longrightarrow L^1(\Gamma_0^2, d\lambda)$ is continuous, $R_t^{\delta} = \mathcal{H}u_t^{\delta}$ is continuous w.r.t.
 $t \geq 0$ on $L^1(\Gamma_0^2, d\lambda)$. Hence $(R_t^{\delta})_{t \geq 0}$ is a weak solution to \eqref{GMCS:31}.
 The main result from \cite{BALL77} therefore implies $R_t^{\delta} = S_{\delta}(t)R_0 \geq 0$.
 Finally, for any $G \in B_{bs}^+(\Gamma_0^2)$ we get
 \[
  \langle G, u_t^{\delta}\rangle = \langle \mathbb{K}_0G, R_t^{\delta}\rangle \geq 0, \quad  t \geq 0.
 \]
\end{proof}

We are now prepared to complete the proof of positive definiteness.
\begin{proof}(Proposition \ref{PHDIPSTH:02})
Let $\mu_0 \in \mathcal{P}_{\alpha',\beta'}$ with correlation function $k_0 \in \K_{\alpha',\beta'}$. Define
\[
 k_{0,\delta}(\eta) := k_0(\eta)e_{\lambda}(R_{\delta};\eta^+)e_{\lambda}(R_{\delta};\eta^-), \quad \delta > 0, \quad \eta \in \Gamma_0^2,
\]
then $k_{0,\delta} \in \mathcal{B}_{\alpha',\beta'}^*$ and it is positive definite, cf. \cite{F11,F11TC}. By Lemma \ref{GMCSLEMMA:01} we get
$\widehat{T}_{\alpha,\beta}^{\delta}(t)^*k_{0,\delta} = U_{\delta}(t)^*k_{0,\delta} \in \mathcal{B}_{\alpha,\beta}^*$ and by Lemma \ref{GMCSLEMMA:02} the latter expression is
positive definite. Let $G \in B_{bs}^+(\Gamma_0^2)$. Then it suffices to show that
\[
 \langle G, \widehat{T}_{\alpha,\beta}^{\delta}(t)^* k_{0,\delta}\rangle \longrightarrow \langle G, \widehat{T}_{\alpha,\beta}(t)^*k_0 \rangle, \quad \delta \to 0.
\]
To this end observe that
\[
 \langle G, \widehat{T}_{\alpha,\beta}^{\delta}(t)^*k_{0,\delta}\rangle = \langle \widehat{T}_{\alpha,\beta}^{\delta}(t)^*G - \widehat{T}_{\alpha,\beta}(t)G, k_{0,\delta}\rangle + \langle \widehat{T}_{\alpha,\beta}(t)G, k_{0,\delta}\rangle.
\]
The first term can be estimated by
\[
 \Vert \widehat{T}_{\alpha,\beta}^{\delta}(t)G - \widehat{T}_{\alpha,\beta}(t)G\Vert_{\Lb_{\alpha,\beta}}\Vert k_0 \Vert_{\K_{\alpha,\beta}}
\]
and hence tends by Lemma \ref{GMCSLEMMA:05} to zero. The second term tends by dominated convergence to $\langle \widehat{T}_{\alpha,\beta}(t)G, k_0\rangle = \langle G, \widehat{T}_{\alpha,\beta}(t)^*k_0\rangle$,
which implies that $\widehat{T}_{\alpha,\beta}(t)^*k_0$ is positive definite.
%\\

If conditions (B) and (C) hold for all $\tau > 0$, then $k_{0, \delta}(\eta) := e^{-\delta|\eta|} k_0(\eta)$ belongs to
$\K_{\alpha - \delta, \beta - \delta}$ for any $\delta > 0$. Consequently, above considerations imply that $\widehat{T}_{\alpha,\beta}(t)^* k_{0,\delta} \in \K_{\alpha,\beta}$
is positive definite. Taking the limit $\delta \to 0$ yields the assertion.
\end{proof}

\begin{Remark}
 Suppose instead of (B) the following to be satisfied: There exist $C > 0$, $N \in \N$ and $\nu_b \geq 0$, $\nu_1, \nu_2 \geq 0$ such that for all $x \in \R^d$ and $\eta \in \Gamma_0^2$:
 \begin{align*}
  b^+(x,\eta) + b^-(x,\eta) &\leq C(1 + |\eta|)^N e^{\nu_b|\eta|},
  \\ d^+(x,\eta) &\leq C(1+|\eta|)^N e^{\nu_1|\eta|},
  \\ d^-(x,\eta) &\leq C (1 + |\eta|)^N e^{\nu_2 |\eta|}.
 \end{align*}
 Then for any positive definite $k_0 \in \K_{\alpha',\beta'}$ the evolution $\widehat{T}_{\alpha,\beta}(t)^*k_0$ is positive definite,
 provided (C) holds for $\alpha' + \nu_1 < \alpha$, $\beta' + \nu_2 < \beta$.
\end{Remark}

\section{Vlasov scaling}

\subsection*{General description of Vlasov scaling}
Let us briefly explain the Vlasov scaling in the two-component case. A motivation and additional explanations can be found in \cite{FKK10}.
Let $L$ be a Markov (pre-)generator on $\Gamma^2$, the aim is to find a scaling $L_n$ such that the following scheme holds.
Let $T_{n}^{\Delta}(t) = e^{tL_{n}^{\Delta}}$ be the (heuristic) representation of the scaled evolution of correlation functions, see \eqref{PHDIPS:11}.
The particular choice of $L \to L_{n}$ should preserve the order of singularity, that is the limit
\begin{align}\label{MCSL:11}
 n^{-|\eta|}T_{n}^{\Delta}(t)n^{|\eta|}k \longrightarrow T_V^{\Delta}(t)k, \quad  n \to 0
\end{align}
should exist and the evolution $T_V^{\Delta}(t)$ should preserve Lebesgue-Poisson exponentials,
i.e. if $r_0(\eta) = e_{\lambda}(\rho_0^{-}, \eta^-)e_{\lambda}(\rho_0^{+};\eta^+)$, then
$T_V^{\Delta}(t)r_0(\eta) = e_{\lambda}(\rho_t^{-},\eta^-)e_{\lambda}(\rho_t^{+};\eta^+)$.
In such a case $\rho_t^{-}, \rho_t^{+}$ satisfy the system of non-linear integro-differential equations
\begin{align}\label{MCSL:06}
 \frac{\partial \rho_t^{-}}{\partial t} = v_-(\rho_t^{-}, \rho_t^{+}), \quad \frac{\partial \rho_t^{+}}{\partial t} = v_+(\rho_t^{-},\rho_t^{+}).
\end{align}
The functionals $v_-, v_+$ can be computed explicitly for a large class of models and \eqref{MCSL:06} is the system of kinetic equations
for the densities $\rho_t^{-}, \rho_t^{+}$ for the particle system. Instead of investigating the limit \eqref{MCSL:11}, we define renormalized operators
$L_{n,\mathrm{ren}}^{\Delta} := n^{-|\eta|} L_{n}^{\Delta}n^{|\eta|}$ and study the behaviour of the semigroups $T_{n,\mathrm{ren}}^{\Delta}(t)$ when $n \to \infty$.
In such a case one can compute a limiting operator
\begin{align}\label{MCSL:12}
 L_{n,\mathrm{ren}}^{\Delta}\longrightarrow L_V^{\Delta}
\end{align}
and show that $L_V^{\Delta}$ is associated to a semigroup $T_V^{\Delta}(t)$. The limit \eqref{MCSL:11} is then obtained by showing the convergence
\begin{align}\label{MCSL:13}
 T_{n,\mathrm{ren}}^{\Delta}(t) \longrightarrow T_V^{\Delta}(t)
\end{align}
in a proper sense.

\subsection*{Scaling of two-component model}
Consider the scaled intensities $d^+_{n}, d^-_{n}, b^+_{n}, b^-_{n} \geq 0$ and suppose they satisfy condition (A) for any $n \in \N$.
Let $L_{n} = L^{-}_{n} + L^{+}_{n}$ where
\begin{align*}
 L^{-}_{n}F(\gamma) &= \sum _{x \in \gamma^-}d^-_{n}(x,\gamma^+, \gamma^- \backslash x)(F(\gamma^+, \gamma^- \backslash x) - F(\gamma^+, \gamma^-))
 \\ & \ \ \ + n \int _{\R^d}b^-_{n}(x,\gamma^+,\gamma^-)(F(\gamma^+, \gamma^- \cup x) - F(\gamma^+, \gamma^-))\,d x
\end{align*}
and
\begin{align*}
 L_{n}^{+}(t)F(\gamma) &= \sum _{x \in \gamma^+}d_{n}^+(x,\gamma^+
 \backslash x, \gamma^-)(F(\gamma^+ \backslash x, \gamma^-) - F(\gamma^+, \gamma^-))
 \\ & \ \ \ + n \int _{\R^d}b_{n}^+(x,\gamma^+,\gamma^-)(F(\gamma^+ \cup x, \gamma^-) - F(\gamma^+, \gamma^-))\,dx.
\end{align*}
Introduce
\begin{align*}
 c_{n}(\alpha,\beta; \eta) :=
  &+ \sum _{x \in \eta^-}\int _{\Gamma_0^2}|\mathbb{K}_0^{-1}d_{n}^-(x,\cdot \cup \eta^+,
  \cdot \cup \eta^- \backslash x)|(\xi)n^{|\xi|}e^{\alpha|\xi^+|}e^{\beta|\xi^-|}\,d\lambda(\xi)
  \\ &\ \sum _{x \in \eta^+}\int _{\Gamma_0^2}|\mathbb{K}_0^{-1}d_{n}^+(x,\cdot \cup \eta^+
  \backslash x,\cdot \cup \eta^-)|(\xi)n^{|\xi|}e^{\alpha|\xi^+|}e^{\beta|\xi^-|}\,d\lambda(\xi)
  \\ &+ e^{-\beta} \sum _{x \in \eta^-} \int _{\Gamma_0}|\mathbb{K}_0^{-1}b_{n}^-
  (x,\cdot\cup \eta^+,\cdot \cup \eta^- \backslash x)|(\xi) n^{|\xi|}e^{\alpha|\xi^+|}e^{\beta|\xi^-|} \,d\lambda(\xi)
  \\ &+ e^{-\alpha}\sum _{x \in \eta^+} \int _{\Gamma_0}|\mathbb{K}_0^{-1}b_{n}^+(x,\cdot \cup \eta^+
  \backslash x,\cdot\cup \eta^-)|(\xi)n^{|\xi|} e^{\alpha|\xi^+|}e^{\beta|\xi^-|}\,d\lambda(\xi)
\end{align*}
and $M_n(\eta) := \sum _{x \in \eta^-}d_n^-(x,\eta^+, \eta^- \backslash x) + \sum _{x \in \eta^+}d_n^+(x,\eta^+\backslash x, \eta^-)$.
We will suppose the following conditions to be satisfied:
\begin{enumerate}
 \item[(V1)] There exists $a(\alpha,\beta) \in (0,2)$ such that for all $\eta \in \Gamma_0^2$ and $n \in \N$
 \[
  c_n(\alpha,\beta;\eta) \leq a(\alpha,\beta)M_n(\eta)
 \]
 is satisfied.
 \item[(V2)] For all $\xi \in \Gamma_0^2$ and $x \in \R^d$ the following limits exist in $\Lb_{\alpha,\beta}$ and are independent of $\xi$
 \begin{align*}
  \lim _{n \to \infty} n^{|\cdot|} (\mathbb{K}_0^{-1}d_{n}^-(x,\cdot \cup \xi)) &= \lim _{n \to \infty} n^{|\cdot|}(\mathbb{K}_0^{-1}d_{n}^{-}(x,\cdot)) =:
  D_x^{V,-},
 \\ \lim _{n \to \infty} n^{|\cdot|} (\mathbb{K}_0^{-1}d_{n}^+(x,\cdot \cup \xi)) &= \lim _{n \to \infty} n^{|\cdot|}(\mathbb{K}_0^{-1}d_{n}^{+}(x,\cdot)) =:
 D_x^{V,+},
 \\ \lim _{n \to \infty} n^{|\cdot|} (\mathbb{K}_0^{-1}b_{n}^-(x,\cdot \cup \xi)) &= \lim _{n \to \infty} n^{|\cdot|}(\mathbb{K}_0^{-1}b_{n}^{-}(x,\cdot)) =:
 B_x^{V,-},
 \\ \lim _{n \to \infty} n^{|\cdot|} (\mathbb{K}_0^{-1}b_{n}^+(x,\cdot \cup \xi)) &= \lim _{n \to \infty} n^{|\cdot|}(\mathbb{K}_0^{-1}b_{n}^{+}(x,\cdot)) =: B_x^{V,+}.
 \end{align*}
 \item[(V3)] Let $M_V(\eta) := \sum _{x \in \eta^+}D_x^+(\emptyset) + \sum _{x \in \eta^-}D_{x}^{-}(\emptyset)$, then there exists $\sigma > 0$ such that either
 \[
  M_n(\eta) \leq \sigma M_V(\eta), \quad \eta \in \Gamma_0^2, \quad  n \in \N
 \]
 or
 \[
  M_n(\eta) \leq \sigma M_V(\eta), \quad \eta \in \Gamma_0^2, \quad n \in \N
 \]
 are satisfied.
\end{enumerate}
Note that here and below the constants $\alpha,\beta$ are fixed.
For simplicity of notation, we omit, whenever it is possible, the dependence on these constants.
Define $\widehat{L}_n := \mathbb{K}_0^{-1}L_n \mathbb{K}_0$ and
the renormalized operators $\widehat{L}_{n,\mathrm{ren}} :=
R_{n}\widehat{L}_{n}R_{n^{-1}}$, where $R_{\kappa}G(\eta) =
\kappa^{|\eta|}G(\eta)$. Then for any $G \in B_{bs}(\Gamma_0^2)$
and $k \in \K_{\alpha,\beta}$ the relation $\langle
\widehat{L}_{n, \mathrm{ren}}G, k \rangle = \langle G,
L^{\Delta}_{n, \mathrm{ren}}k \rangle$ holds.

\subsection*{Statements}
The next statement provides the existence and uniqueness of an evolution of quasi-observables and correlation functions for any fixed $n \in \N$.
\begin{Theorem}\label{MCSLTH:00}
 Suppose that condition (V1) is satisfied. Then for any fixed $n \in \N$ the following assertions are true:
\begin{enumerate}
  \item The closure of $(\widehat{L}_{n,\mathrm{ren}}, B_{bs}(\Gamma_0^2))$ is given by $(\widehat{L}_{n,\mathrm{ren}}, D(\widehat{L}_{n,\mathrm{ren}}))$,
  where
  \[
   D(\widehat{L}_{n,\mathrm{ren}}) = \{ G \in \Lb_{\alpha,\beta}\ | \ M_{n}\cdot G \in \Lb_{\alpha,\beta} \}.
  \]
  It is the generator of an analytic semigroup $(\widehat{T}_{n,\mathrm{ren}}(s))_{s \geq 0}$ of contractions on $\Lb_{\alpha,\beta}$.
  \item Let $\widehat{T}_{n, \mathrm{ren}}(t)^*$ be the adjoint semigroup with $(L^{\Delta}_{n, \mathrm{ren}}, D(L^{\Delta}_{n, \mathrm{ren}}))$
  considered on the (maximal) domain
  \[
   D(L^{\Delta}_{n,\mathrm{ren}}) = \{ k \in \K_{\alpha,\beta} \ | \ L^{\Delta}_{n, \mathrm{ren}}k \in \K_{\alpha,\beta}\}.
  \]
  For any $n \in \N$ and $k_0 \in \K_{\alpha,\beta}$, there exists a unique weak solution to
  \[
   \frac{\partial }{\partial t}\langle G, k_{t,n}\rangle = \langle \widehat{L}_{n,\mathrm{ren}}G, k_{t,n}\rangle,
   \quad  k_{t,n}|_{t=0} = k_0, \quad  G \in B_{bs}(\Gamma_0^2)
  \]
  given by $k_{t,n} = \widehat{T}_{n, \mathrm{ren}}(t)^*k_0$.
 \end{enumerate}
\end{Theorem}

The case $n=1$ is covered by the results obtained in Theorem \ref{GMCSTH:01}.
Following the arguments there, it is not difficult to adopt the proofs to this case.
In the next step we construct the limiting dynamics when $n \to \infty$. Condition (V2) suggests to consider the limit
$\widehat{L}_{n,\mathrm{ren}}G \longrightarrow \widehat{L}_{V}G$, as $n \to \infty$.
The operator $\widehat{L}_V := A_V + B_V$ is given by $A_VG(\eta) = - M_V(\eta)G(\eta)$, where
\begin{align*}
 M_V(\eta) = &\ \sum _{x \in \eta^+}D_x^{V,+}(\emptyset) + \sum _{x \in
 \eta^-}D_x^{V,-}(\emptyset),
 \\ B_VG(\eta) = &- \sum _{\bfrac{\xi^+ \subsetneq \eta^+}{\xi^- \subsetneq \eta^-}}G(\xi)\sum _{x \in \xi^+}D_x^{V,+}(\eta \backslash \xi)
 - \sum _{\bfrac{\xi^+ \subsetneq \eta^+}{\xi^- \subsetneq \eta^-}}G(\xi)\sum _{x \in \xi^-}D_x^{V,-}(\eta \backslash \xi)
 \\ &+ \sum _{\xi \subset \eta}\int _{\R^d}G(\xi^+ \cup x, \xi^-)B_x^{V,+}(\eta \backslash \xi)\,dx
 + \sum _{\xi \subset \eta}\int _{\R^d}G(\xi^+,\xi^- \cup x)B_x^{V,-}(\eta \backslash \xi)\,dx.
\end{align*}
Let $D(\widehat{L}_V) := \{ G \in \Lb_{\alpha,\beta}\ | \ M_V\cdot G \in \Lb_{\alpha,\beta}\}$, define
\begin{align*}
 & c_V(\alpha, \beta;\eta) :=
 \\ &+ \sum _{x \in \eta^+}\int _{\Gamma_0^2}|D_x^{V,+}(\xi)|e^{\alpha|\xi^+|}e^{\beta|\xi^-|}\,d\lambda(\xi)
     + e^{-\alpha}\sum _{x \in \eta^+}\int _{\Gamma_0^2}|B_x^{V,+}(\xi)|e^{\alpha|\xi^+|}
     e^{\beta|\xi^-|}\,d\lambda(\xi)
 \\ &+ \sum _{x \in \eta^-}\int _{\Gamma_0^2}|D_x^{V,-}(\xi)|e^{\alpha|\xi^+|}e^{\beta|\xi^-|}\,d\lambda(\xi)
     + e^{-\beta}\sum _{x \in \eta^-}\int _{\Gamma_0^2}|B_x^{V,-}(\xi)|e^{\alpha|\xi^+|}e^{\beta|\xi^-|}\,d\lambda(\xi)
\end{align*}
and finally
\begin{align*}
 & (L_V^{\Delta}k)(\eta) :=
 \\ &- \sum _{x \in \eta^+}\int _{\Gamma_0^2}k(\eta \cup \xi)D_x^{V,+}(\xi)\,d\lambda(\xi)
     + \sum _{x \in \eta^+}\int _{\Gamma_0^2}k(\eta^+ \backslash x \cup \xi^+, \eta^+ \cup \xi^+)B_x^{V,+}
     (\xi)\,d\lambda(\xi)
 \\ &- \sum _{x \in \eta^-}\int _{\Gamma_0^2}k(\eta \cup \xi)D_x^{V,-}(\xi)\,d\lambda(\xi)
     + \sum _{x \in \eta^-}\int _{\Gamma_0^2}k(\eta^+ \cup \xi^+, \eta^-\backslash x\cup \xi^-)B_x^{V,-}
     (\xi)\,d\lambda(\xi).
\end{align*}
\begin{Theorem}\label{MCSLTH:01}
 Assume that conditions (V1), (V2) are satisfied. Then the following assertions are true:
 \begin{enumerate}
  \item The operator $(\widehat{L}_V, D(\widehat{L}_V))$ is the generator of an analytic semigroup $(\widehat{T}^V(t))_{t \geq 0}$ of contractions
  on $\Lb_{\alpha,\beta}$.
  \item Let $(\widehat{T}^V(t)^*)_{t \geq 0}$ be the adjoint semigroup on $\K_{\alpha,\beta}$, then for any $r_0 \in \K_{\alpha,\beta}$
  there exists a unique solution $r_t = \widehat{T}^V(t)^*r_0$ to the Cauchy problem
  \begin{align}\label{MCSL:08}
   \frac{\partial }{\partial t}\langle G, r_{t}\rangle = \langle \widehat{L}_VG, r_t \rangle, \quad r_t|_{t=0} = r_0, \quad G \in B_{bs}(\Gamma_0^2).
  \end{align}
  \item  Let $r_0(\eta) = \prod _{x \in \eta^+}\rho_0^{+}(x) \prod _{x \in \eta^-}\rho_0^{-}(x)$ and $\rho_0^{+},\rho_0^{-} \in L^{\infty}(\R^d)$ with
  $\Vert \rho_0^{+} \Vert_{L^{\infty}} \leq e^{\alpha}$, $\Vert \rho_0^{-} \Vert_{L^{\infty}} \leq e^{\beta}$.
  Assume that $(\rho_t^{+}, \rho_t^{-})$ is a classical solution to
  \begin{align*}
   \frac{\partial \rho_t^{-}}{\partial t}(x) = &- \int _{\Gamma_0^2}e_{\lambda}(\rho_t^{+};\xi^+)e_{\lambda}
   (\rho_t^{-};\xi^-)D_x^{V,-}(\xi)\,d\lambda(\xi)\rho_t^{-}(x)
   \\ &+ \int
   _{\Gamma_0^2}e_{\lambda}(\rho_t^{+};\xi^+)e_{\lambda}(\rho_t^{-};\xi^-)B_x^{V,-}(\xi)\,d\lambda(\xi),
   \\ \frac{\partial \rho_t^{+}}{\partial t}(x) = &- \int _{\Gamma_0^2}e_{\lambda}(\rho_t^{+};\xi^+)e_{\lambda}(\rho_t^{-};\xi^-)D_x^{V,+}(\xi)\,d\lambda(\xi)\rho_t^{+}(x)
   \\ &+ \int _{\Gamma_0^2}e_{\lambda}(\rho_t^{+};\xi^+)e_{\lambda}(\rho_t^{-};\xi^-)B_x^{V,+}(\xi)\,d\lambda(\xi)
  \end{align*}
  with initial conditions $\rho_t^{+}|_{t=0} = \rho_0^{+}$, $\rho_t^{-}|_{t=0} = \rho_0^{-}$ and
  \[
   \Vert \rho_t^{+} \Vert_{L^{\infty}} \leq e^{\alpha} \quad \Vert \rho_t^{-}\Vert_{L^{\infty}} \leq e^{\beta}.
  \]
 Then $r_t(\eta) := \prod _{x \in \eta^+}\rho_t^{+}(x) \prod _{x \in \eta^-}\rho_t^{-}(x)$
  is a weak solution to \eqref{MCSL:08} in $\K_{\alpha,\beta}$.
 \end{enumerate}
\end{Theorem}
\begin{proof}
 By conditions (V1) and (V2) it follows that $c_V(\alpha,\beta;\eta)  \leq a(\alpha,\beta)M_V(\eta)$ holds.
 Define a positive operator $B'_V$ on $D(\widehat{L}_V)$ by
 \begin{align*}
 B_V'G(\eta) = & \sum _{\bfrac{\xi^+ \subsetneq \eta^+}{\xi^- \subsetneq \eta^-}}G(\xi)\sum _{x \in \xi^+}|D_x^{V,+}(\eta \backslash \xi)|
 + \sum _{\bfrac{\xi^+ \subsetneq \eta^+}{\xi^- \subsetneq \eta^-}}G(\xi)\sum _{x \in \xi^-}|D_x^{V,-}(\eta \backslash \xi)|
 \\ &+ \sum _{\xi \subset \eta}\int _{\R^d}G(\xi^+ \cup x, \xi^-)|B_x^{V,+}(\eta \backslash \xi)|\,dx\\
 &+ \sum _{\xi \subset \eta}\int _{\R^d}G(\xi^+,\xi^- \cup x)|B_x^{V,-}(\eta \backslash \xi)|\,dx.
 \end{align*}
 Then it is not difficult to see that for any $0 \leq G \in D(\widehat{L}_V)$
 \begin{align*}
  \int _{\Gamma_0^2}B_V'G(\eta) e^{\alpha|\eta^+|}e^{\beta|\eta^-|}\,d\lambda(\eta)
  \leq (a(\alpha,\beta) - 1)\int _{\Gamma_0^2}M_V(\eta)G(\eta)e^{\alpha|\eta^+|}e^{\beta|\eta^-|}\,d\lambda(\eta)
 \end{align*}
 is fulfilled.
 The same arguments as in the proof of Theorem \ref{MCSLTH:00} yield existence, analyticity and the contraction property of the semigroup $\widehat{T}^V(t)$.
 For the last assertion we only show that $r_t$ is continuous w.r.t. $\mathcal{C}$. The other assertions are simple computations, see e.g. \cite{FKK10}.
 First observe that by $|r_t(\eta)| \leq e^{\alpha|\eta^+|}e^{\beta|\eta^-|}$ the function $r_t$ is norm-bounded and hence it suffices to show that it is continuous
 w.r.t. $\sigma(\mathcal{K}_{\alpha,\beta}, \Lb_{\alpha,\beta})$. But this function is continuous in $t \geq 0$ for any $\eta$ and hence the assertion follows
 by dominated convergence.
\end{proof}
\begin{Theorem}\label{MCSLTH:02}
 Suppose that conditions (V1)--(V3) are fulfilled. Then
 $\widehat{T}_{n, \mathrm{ren}}(t) \longrightarrow \widehat{T}^V(t)$ holds strongly in $\Lb_{\alpha,\beta}$ and uniformly on compacts in $t \geq 0$.
\end{Theorem}
\begin{proof}
 We are going to apply \cite[Lemma 4.3]{FKK12} and Trotter-Kato approximation. Fix $\lambda > 0$ and denote by $R(\lambda; A_{n})$ and $R(\lambda, A_V)$ the resolvent
 for $A_{n}$ and $A_V$, respectively. Then it follows that
 $\Vert R(\lambda;A_{n})\Vert_{L(\Lb_{\alpha,\beta})},\ \Vert R(\lambda;A_V) \Vert_{L(\Lb_{\alpha,\beta})} \leq \frac{1}{\lambda}$,
 $\Vert B_{n}R(\lambda;A_{n})G\Vert_{\Lb_{\alpha,\beta}} \leq (a(\alpha,\beta) - 1)\Vert G \Vert_{\Lb_{\alpha,\beta}}$
 and likewise $\Vert B_{V}R(\lambda;A_V)G\Vert_{\Lb_{\alpha,\beta}} \leq (a(\alpha,\beta) - 1)\Vert G \Vert_{\Lb_{\alpha,\beta}}$.
 Since $M_{n} \longrightarrow M_{V}$ as $n \to \infty$, it is easy to show by dominated convergence that $R(\lambda;A_{n}) \longrightarrow R(\lambda;A_V)$
 holds strongly in $\Lb_{\alpha,\beta}$ as $n \to \infty$. Hence it remains to show the convergence
 \begin{align}\label{MCSL:14}
  B_{n}R(\lambda;A_{n})G \longrightarrow B_VR(\lambda;A_{V})G, \quad n \to \infty.
 \end{align}
 This can be proved, similarly to \cite{FKK12}, by dominated convergence.
\end{proof}

\begin{Remark}
 The proof shows that condition (V3) can be replaced by
 \[
  d_n^-(x,\eta) + d_n^+(x,\eta) \leq C(1 + |\eta|)^N e^{\tau|\eta|}, \quad x \in \R^d, \quad \eta \in \Gamma_0^2, \quad n \in \N
 \]
 for some constants $C > 0$, $N \in \N$ and $\tau \geq 0$.
\end{Remark}

\section{Examples}
We introduce four models describing the stochastic behaviour of particle systems on $\Gamma^2$.
Interactions of particles, of the same and also of different type, are modelled by the relative energy function
\[
 E_{\varphi}(x,\gamma^{\pm}) := \sum _{y \in \gamma^{\pm}}\varphi(x-y), \quad x \in \R^d, \quad \gamma^{\pm} \in \Gamma,
\]
where $\varphi$ is a symmetric, non-negative and integrable function.
Associated to above relative energy is the following functional
\begin{align}\label{GMCS:30}
 C(\varphi) := \int _{\R^d}|e^{-\varphi(x)} - 1|\,dx.
\end{align}
Since $\varphi$ is non-negative, we obtain $C(\varphi) \leq \int
_{\R^d}\varphi(x)\,dx =: \langle \varphi \rangle$. Note
that below we will give no proofs of the statements. Conditions
(A)--(C) are rather standard and can be checked by similar methods
to \cite{FKK12, FKK15}. The last condition can be checked
similarly to \cite{F16, FK16}. Note that a general criterion for
condition (D) can be also found in \cite{TV06}.

\subsection{Two-interacting BDLP-model}
Suppose that the death intensities are given by
\begin{align*}
 d^-(x,\gamma^+, \gamma^- \backslash x) &= m^- + \sum _{y \in \gamma^- \backslash x}
 a^-(x-y),
 \\ d^+(x,\gamma^+ \backslash x, \gamma^-) &= m^+ + \sum _{y \in \gamma^+ \backslash x} b^-(x-y) + \sum _{y \in \gamma^-}\varphi^-(x-y).
\end{align*}
This means that the particles in $\gamma^{\pm}$ have a random lifetime determined by the parameters $m^+, m^- > 0$.
The additional terms describe the competition of particles for resources. Namely, each particle $x \in \gamma^-$ may die due to the competition for resources with
another particle $y \in \gamma^- \backslash x$ from the same type. The rate of this event is determined by $a^-(x-y)$.
Likewise each particle $x \in \gamma^+$ may be die due to the interaction with another particle by $y \in \gamma^+ \backslash x$ with the rate $b^-(x-y)$.
Moreover, this particle $x \in \gamma^+$ may also die due to the competition for resources with another particle $y \in \gamma^-$ of different type.
The rate of this event is described by the interaction potential $\varphi^-(x-y)$.

The birth intensities are assumed to be given by
\begin{align*}
 b^-(x,\gamma) &= \sum _{y \in \gamma^-}a^+(x-y) + z,
 \\ b^+(x,\gamma) &= \sum _{y \in \gamma^+}b^+(x-y) + \sum _{y \in \gamma^-}\varphi^+(x-y).
\end{align*}
The terms corresponding to $a^+$ and $b^+$ describe the free branching of particles $\gamma^{\pm}$ independently of each other.
Additionally, each particle $y \in \gamma^-$ may create a new particle at position $x \in \R^d$ of opposite type. The spatial distribution of the new particle
is determined by the function $\varphi^+$. By definition, such events occur independently of each other.
The term including $z \geq 0$ describes the additional creation of particles by an outer source.
Suppose that $a^{\pm}, b^{\pm}, \varphi^{\pm}$ are non-negative, symmetric and integrable.
\begin{Theorem}
 Suppose that $a^{\pm}, b^{\pm}, \varphi^{\pm}$ are bounded and there exist constants $b_1, b_2 \geq 0$ and $\vartheta_1, \vartheta_2, \vartheta_3 > 0$ such that
 \begin{align*}
  \sum _{x \in \eta^+}\sum _{y \in \eta^+ \backslash x}b^+(x-y) &\leq \vartheta_1
  \sum _{x \in \eta^+}\sum _{y \in \eta^+ \backslash x}b^-(x-y) +
  b_1|\eta^+|,
  \\ \sum _{x \in \eta^-}\sum _{y \in \eta^- \backslash x}a^+(x-y) &\leq \vartheta_2
  \sum _{x \in \eta^-}\sum _{y \in \eta^- \backslash x}a^-(x-y) + b_2|\eta^-|,
 \end{align*}
 and $\varphi^+ \leq \vartheta_3 \varphi^-$ hold. Moreover, assume that the parameters satisfy the relations $\vartheta_1, \vartheta_3 < e^{\alpha}$, $\vartheta_2 < e^{\beta}$,
 \begin{align*}
  m^+ &> e^{\alpha}\langle b^- \rangle + e^{\beta}\langle \varphi^- \rangle + e^{-\alpha}b_1 + \langle b^+\rangle +
  \langle \varphi^+ \rangle ,
 \\ m^- &> e^{\beta}\langle a^- \rangle + e^{-\beta}(b_2 + z) + \langle a^+\rangle.
 \end{align*}
 Then conditions (A)--(D) hold for $\tau = 0$ and \eqref{PHDIPS:02} is fulfilled.
\end{Theorem}
\begin{Remark}
 Note that the one-component case has been considered in \cite{FKK12, KK16}.
\end{Remark}
Suppose that the conditions given above are fulfilled. Then
(V1)--(V3) are satisfied and after Vlasov scaling we arrive at the
kinetic equations
\begin{align*}
 \frac{\partial \rho_t^{-}}{\partial t}(x) = &- m^- \rho_t^{-}(x) - \rho_t^{-}(x)(a^- \ast \rho_t^{-})(x) + (a^+
 \ast \rho_t^{-})(x) + z,
 \\ \frac{\partial \rho_t^{+}}{\partial t}(x) = &- \left( m^+ + (\varphi^- \ast \rho_t^{-})(x)\right) \rho_t^{+}(x) - \rho_t^{+}(x)(b^- \ast \rho_t^{+})(x)
 \\ &+ (b^+\ast \rho_t^{+})(x) + (\varphi^+ \ast \rho_t^{-})(x).
\end{align*}
Here $(f \ast g)(x) := \int _{\R^d}f(x-y)g(y)\,dy$ denotes
the usual convolution of functions on $\R^d$.

\subsection{Two interacting Glauber-models}
Suppose that the death intensities are given by
\begin{align*}
 d^-(x,\gamma^+, \gamma^- \backslash x) &=
 \exp\left(-sE_{\psi^+}(x,\gamma^+)\right),
 \\ d^+(x,\gamma^+ \backslash x, \gamma^-) &= \exp\left(-sE_{\psi^-}(x,\gamma^-)\right),
\end{align*}
where $s \in [0,\frac{1}{2}]$ and $\psi^+, \psi^-$ are symmetric, non-negative and integrable.
The birth intensities are assumed to be of the form
\begin{align*}
 b^-(x, \gamma) &= z^- \exp\left(-(1-s)E_{\psi^+}(x, \gamma^+)\right)
 \exp\left(-E_{\phi^-}(x,\gamma^-)\right),
 \\ b^+(x, \gamma) &= z^+ \exp\left( -(1-s)E_{\psi^-}(x, \gamma^-)\right) \exp\left(-E_{\phi^+}(x,\gamma^+)\right),
\end{align*}
where $z^-, z^+ > 0$ and $\phi^-, \phi^+$ are assumed to be non-negative, symmetric and integrable.
This model describes the time-evolution of two interacting types of particles. In contrast to the previous model,
the branching of this particles is replaced by an additional source. The next theorem provides an evolution of states.
\begin{Theorem}
 Let $\phi^+, \phi^-, \psi^+, \psi^-$ be symmetric, non-negative and integrable and assume that the paramters satisfy the relations
 \begin{align}
  \label{GMCS:61} e^{e^{\alpha}C(s\psi^+)} + e^{-\beta}z^- e^{e^{\alpha}C((1-s)\psi^+)} e^{ e^{\beta}C(\phi^-)} &<
  2,
  \\ \label{GMCS:62} e^{e^{\beta}C(s\psi^-)} + e^{-\alpha}z^+ e^{e^{\beta} C((1-s)\psi^-)} e^{e^{\alpha}C(\phi^+)} &< 2.
 \end{align}
 Then conditions (A)--(D) are satisfied for $\tau = 0$ and \eqref{PHDIPS:02} holds.
\end{Theorem}

In the case $s=0$ conditions \eqref{GMCS:61} and \eqref{GMCS:62} simplify to
\begin{align}
 \label{GMCS:54} z^- e^{e^{\alpha}C(\psi^+)} e^{ e^{\beta}C(\phi^-)} &<
 e^{\beta},
 \\ \label{GMCS:55} z^+ e^{e^{\beta} C(\psi^-)} e^{e^{\alpha}C(\phi^+)} &< e^{\alpha}.
\end{align}
Of particular interest is the special case $\phi^+ = 0 = \phi^-$,
also known as the Widom-Rowlinson model. The non-equilibrium
dynamics for this model has recently been ana\-lysed in
\cite{FKK15WRMODEL}, but without conditions \eqref{GMCS:54} and
\eqref{GMCS:55} only existence of a local evolution of correlation
functions could have been shown. Conditions \eqref{GMCS:54} and
\eqref{GMCS:55} are satisfied for $e^{-\alpha} = C(\psi^+)$ and
$e^{-\beta} = C(\psi^-)$ if
\[
 z^- < \frac{1}{eC(\psi^-)} \quad \text{and} \quad  z^+ < \frac{1}{eC(\psi^+)}
\]
are satisfied. An extension of such model with density dependent mutation rates (i.e. particles can change their types) has been recently
considered in \cite{F16b}.

\subsubsection*{Vlasov scaling}
For simplicity we consider the case $s = 0$, hence the death intensities need not to be scaled, i.e. are given by
\[
 d^-(x,\gamma^+, \gamma^- \backslash x) = 1 = d^+(x,\gamma^+ \backslash x, \gamma^-).
\]
The scaled birth intensities are given by
\begin{align*}
 b_{n}^-(x, \gamma) &= z^- \exp\left(-\frac{1}{n} E_{\psi^+}(x, \gamma^+)\right) \exp\left(-\frac{1}{n} E_{\phi^-}
 (x,\gamma^-)\right),
 \\ b_{n}^+(x, \gamma) &= z^+ \exp\left( -\frac{1}{n} E_{\psi^-}(x, \gamma^-)\right) \exp\left(-\frac{1}{n} E_{\phi^+}(x,\gamma^+)\right).
\end{align*}
All previous results can be applied which yields the mesoscopic equations, cf. \eqref{MCSL:06}
\begin{align}
 \label{TWOCOMP:01} \frac{\partial \rho_t^{-}}{\partial t}(x) &= -\rho_t^{-}(x) + z^- e^{-(\phi^- \ast \rho_t^{-})(x)}
 e^{-(\psi^+ \ast \rho_t^{+})(x)},
 \\ \label{TWOCOMP:02} \frac{\partial \rho_t^{+}}{\partial t}(x) &= - \rho_t^{+}(x) + z^+ e^{-(\phi^+ \ast \rho_t^{+})(x)} e^{-(\psi^- \ast \rho_t^{-})(x)}.
\end{align}

\subsection{BDLP-model in Glauber environment}
Let us consider death intensities given by
\begin{align*}
 d^-(x,\gamma^+, \gamma^- \backslash x) &= 1,
 \\ d^+(x,\gamma^+ \backslash x, \gamma^-) &= m^+ + \sum _{y \in \gamma^+ \backslash x}a^-(x-y) + \sum _{y \in \gamma^-}\phi(x-y),
\end{align*}
where $m^+ > 0$ and $0 \leq a^-, \phi \in L^1(\R^d)$ are symmetric. The birth intensities are assumed to be of the form
\begin{align*}
 b^-(x, \gamma) &= z^- \exp\left(- E_{\psi}(x,\gamma^-)\right),
 \\ b^+(x, \gamma) &= \sum _{y \in \gamma^+}a^+(x-y) + \sum _{y \in \gamma^-}b^+(x-y),
\end{align*}
where $z^- > 0$ and $0 \leq \psi, a^+, b^+ \in L^1(\R^d)$ are symmetric.
\begin{Theorem}
 Suppose that $a^{\pm}, b^+, \phi$ are bounded and there exist $\theta \in (0, e^{\alpha})$ and $b \geq 0$ such that
 \begin{align}\label{MCSL:72}
  \sum _{x \in \eta^+} \sum _{y \in \eta^+ \backslash x}a^+(x-y) \leq \theta \sum _{x \in \eta^+}\sum _{y \in \eta^+ \backslash x}a^-(x-y) + b|\eta^+|
 \end{align}
 is satisfied. Moreover, assume that for some $\vartheta \in (0, e^{\alpha})$ and
 \begin{align}
  \label{MCSL:73} \vartheta \phi &\geq b^+ ,
  \\ \label{MCSL:74} e^{\beta} &> z^-
  \exp\left(e^{\beta}C(\psi)\right),
  \\ \label{MCSL:75} m^+ &> e^{\alpha}\langle a^-\rangle  + e^{\beta}\langle \phi \rangle  + \langle a^+\rangle  + \langle b^+ \rangle + e^{-\alpha}b
 \end{align}
 hold. Then conditions (A)--(D) are satisfied with $\tau = 0$ and \eqref{PHDIPS:02} holds.
\end{Theorem}

\subsubsection*{Vlasov scaling}
Suppose that $a^{\pm}, b^+, \phi, \psi$ are bounded and
\eqref{MCSL:72}--\eqref{MCSL:75} with  $z^- e^{e^{\beta}\langle
\psi \rangle}< e^{\beta}$ instead of \eqref{MCSL:74} hold. Scaling
of the potentials by $\frac{1}{n}$ yields for the death
\begin{align*}
 d^-(x,\gamma^+, \gamma^- \backslash x) &= 1,
 \\ d_{n}^+(x,\gamma^+ \backslash x, \gamma^-) &= m^+ +  \frac{1}{n} \sum _{y \in \gamma^+ \backslash x}a^-(x-y) + \frac{1}{n} \sum _{y \in \gamma^-}\phi(x-y).
\end{align*}
For the birth we obtain
\begin{align*}
 b_{n}^-(x, \gamma) &= z^- \exp\left(- \frac{1}{n}
 E_{\psi}(x,\gamma^-)\right),
 \\ b_{n}^+(x, \gamma) &= \frac{1}{n} \sum _{y \in \gamma^+}a^+(x-y) + \frac{1}{n} \sum _{y \in \gamma^-}b^+(x-y).
\end{align*}
Taking $n \to \infty$ yields $D_x^{V,-}(\eta) = 0^{|\eta|}$ and
\[
 D_x^{V,+}(\eta) = 0^{|\eta|}m^+ + 0^{|\xi^-|}\one_{\Gamma^{(1)}}(\xi^+)\sum _{y \in \xi^+}a^-(x-y) + 0^{|\xi^+|}\one_{\Gamma^{(1)}}(\xi^-)\sum _{y \in \xi^-}\phi(x-y).
\]
We obtain for the birth intensities
\begin{align*}
 B_x^{V,-}(\eta) &= z^- e_{\lambda}\left( -\psi(x-\cdot);\xi^-\right) 0^{|\xi^+|},
 \\ B_x^{V,+}(\eta) &= 0^{|\xi^-|}\one_{\Gamma^{(1)}}(\xi^+)\sum _{y \in \xi^+}a^+(x-y) + 0^{|\xi^+|} \one_{\Gamma^{(1)}}(\xi^-)\sum _{y \in \xi^-}b^+(x-y).
\end{align*}
The kinetic equation is therefore given by
\begin{align*}
 \frac{\partial \rho_t^{-}}{\partial t}(x) = &- \rho_t^{-}(x) + z^- e^{-(\psi \ast
 \rho_t^{-})(x)},
 \\ \frac{\partial \rho_t^{+}}{\partial t}(x) = &- \left( m^+ + (\phi \ast \rho_t^{-})(x)\right) \rho_t^{+}(x) - \rho_t^{+}(x)(a^- \ast \rho_t^{+})(x)
 \\ & \qquad \qquad  \qquad \qquad \qquad  \qquad \ + (a^+ \ast \rho_t^{+})(x) + (b^+ \ast \rho_t^{-})(x).
\end{align*}

\subsection{Density dependent branching in Glauber environment}
Suppose that the death intensities are given by
\begin{align*}
 d^-(x,\gamma^+, \gamma^- \backslash x) &= 1,
 \\ d^+(x,\gamma^+ \backslash x, \gamma^-) &= m^+\exp\left( E_{\phi^+}(x,\gamma^+ \backslash x)\right),
\end{align*}
where $m^+ > 0$. The birth intensities are given by
\begin{align*}
 b^-(x, \gamma) &= z^- \exp\left(-E_{\phi^-}(x,\gamma^-)\right),
 \\ b^+(x, \gamma) &= \sum _{y \in \gamma^+}\exp\left(-E_{\psi^-}(y, \gamma^-)\right) a^+(x-y)
\end{align*}
with $z^- > 0$ and $a^+, \phi^-, \phi^+, \psi^-$ symmetric, non-negative and integrable.
\begin{Theorem}\label{PHDIPSTH:13}
 Suppose that $0 \neq \phi^+, a^+$ are bounded, there exist constants $\kappa > 0$ and $b \geq 0$ such that for all $\eta^+ \in \Gamma_0$
 \begin{align}\label{GMCS:66}
  \sum _{x \in \eta^+}\sum _{y \in \eta^+ \backslash x}a^+(x-y) \leq \vartheta \sum _{x \in \eta^+}\sum _{y \in \eta^+ \backslash x}\phi^+(x-y) + b|\eta^+|
 \end{align}
 and the parameters satisfy the relations
 \begin{align*}
  e^{\beta} & > z^- \exp \left( e^{\beta}C(\phi^-)\right),
  \\ 2 & > e^{e^{\alpha}C(-\phi^+)} + \frac{\max\{ \langle a^+ \rangle + be^{-\alpha}, \vartheta e^{-\alpha}\}}{m^+} e^{e^{\beta}C(\psi^-)}.
 \end{align*}
 Then conditions (A)--(D) hold with $\tau = \Vert \phi^+ \Vert_{\infty}$ and \eqref{PHDIPS:02} is satisfied.
\end{Theorem}

\subsubsection*{Vlasov scaling}
Scaling all potentials by $\frac{1}{n}$ gives $d^-(x,\gamma^+, \gamma^- \backslash x) = 1$,
\[
 d^+(x,\gamma^+ \backslash x, \gamma^-) = m^+\exp\left( \frac{1}{n}E_{\phi^+}(x,\gamma^+ \backslash x)\right)
\]
and for the birth intensities
\begin{align*}
 b^-(x, \gamma) &= z^- \exp\left(-
 \frac{1}{n}E_{\phi^-}(x,\gamma^-)\right),
 \\ b^+(x, \gamma) &= \frac{1}{n}\sum _{y \in \gamma^+}\exp\left(- \frac{1}{n}E_{\psi^-}(y, \gamma^-)\right) a^+(x-y).
\end{align*}
Suppose that $0 \neq \phi^+, a^+, \phi^-, \psi^-, a^+$ are bounded, \eqref{GMCS:66} holds and the parameters satisfy the stronger relations
 \begin{align*}
  e^{\beta} &> z^- \exp \left( e^{\beta}\langle \phi^-\rangle
  \right),
  \\ 2 &> e^{e^{\alpha}\langle \phi^+\rangle } + \frac{\max\{ \langle a^+ \rangle + be^{-\alpha}, \vartheta e^{-\alpha}\}}{m^+} e^{e^{\beta}\langle \psi^-\rangle}.
 \end{align*}
Then conditions (V1)--(V3) are satisfied. This yields the kinetic
equations
\begin{align*}
 \frac{\partial \rho_t^{-}}{\partial t}(x) = &- \rho_t^{-}(x) + z^- e^{-(\phi^- \ast
 \rho_t^{-})(x)},
 \\ \frac{\partial \rho_t^{+}}{\partial t}(x) = &- m^+ \rho_t^{+}(x)e^{(\phi^+ \ast \rho_t^{+})(x)} + (a^+ \ast \rho_t^{+})(x)e^{-(\psi^- \ast \rho_t^{-})(x)}.
\end{align*}

{\it{Acknowledgments}}. The financial support through the CRC 701
within the project~A5 is gratefully acknowledged. The authors
would like to thank the reviewer for many critical remarks leading
to a significantly better presentation of this work.

%\end{document}

%%%%%%%%%%%%%%%%%%%%%%%%%%%%%%%%%%%%%%%%%%

\end{document}